\documentclass[11pt]{amsart}
\usepackage{amsmath, amssymb, amsfonts, verbatim, graphicx, amsthm, tikz-cd, mathrsfs, color, comment, fancyhdr, amsrefs}
\usepackage[top=1in, bottom=1in, left=1in, right=1in]{geometry}
\usepackage{enumerate}

\usepackage{mathtools}

\usepackage{url}

\usepackage{enumitem}

\usepackage[symbol]{footmisc}

\usepackage[foot]{amsaddr}

\allowdisplaybreaks[1]

\usepackage[sc]{mathpazo}
\usepackage[T1]{fontenc}

\newcommand{\vertiii}[1]{{\left\vert\kern-0.25ex\left\vert\kern-0.25ex\left\vert #1
		\right\vert\kern-0.25ex\right\vert\kern-0.25ex\right\vert}}

\def\XXint#1#2#3{{\setbox0=\hbox{$#1{#2#3}{\int}$ }
		\vcenter{\hbox{$#2#3$ }}\kern-.6\wd0}}

\theoremstyle{plain}
\newtheorem{Thm}{Theorem}[section]
\newtheorem{Prop}[Thm]{Proposition}
\newtheorem{Lem}[Thm]{Lemma}
\newtheorem{Cor}[Thm]{Corollary}

\newtheorem{Qn}{Question}
\newtheorem{Conj}[Thm]{Conjecture}
\newtheorem*{namedthm}{\namedthmname}
\newcounter{namedthm}

\newenvironment{subproof}[1][\proofname]{%
	\begin{proof}[#1]%
	}{%
	\end{proof}%
}



\theoremstyle{definition}
\newtheorem{Def}[Thm]{Definition}
\newtheorem{Not}[Thm]{Notation}
\newtheorem{Rmk}[Thm]{Remark}

\newtheorem{Conv}[Thm]{Convention}

\newcommand{\Z}{\mathbb{Z}}

\renewcommand{\epsilon}{\varepsilon}

\usepackage{setspace}

\usepackage{hyperref}

\author{Aidan Young}
\address{Ben-Gurion University of the Negev, Israel}

\title{Adversarial ergodic optimization}

\email{\url{youngaid@post.bgu.ac.il}}

\keywords{Ergodic optimization, symbolic dynamics, shift of finite type, specification, error-correcting code}
\subjclass[2020]{37B02, 37B10, 37D35}
\begin{document}

\maketitle

\begin{abstract}
In this article, we introduce an ergodic optimization problem inspired by information theory, which can be presented informally as follows: given a factor map $\pi : (X, T) \to (Y, S)$ of topological dynamical systems, and a continuous function $f \in C(X)$, what can be said about the extrema
$$\sup_{y \in Y} \inf_{x \in \pi^{-1} \{y\}} \lim_{k \to \infty} \frac{1}{k} \sum_{j = 0}^{k - 1} f \left( T^j x \right) ? $$
\end{abstract}

Ergodic optimization is concerned with the following (informal) question: given a topological dynamical system $(Y, S)$, and a continuous function $f \in C(Y)$, what can be said about $\sup_{y \in Y} \lim_{k \to \infty} \frac{1}{k} \sum_{j = 0}^{k - 1} f \left( S^j y \right)$? This problem as stated is ill-posed, since the limit in question will generally not exist for all $y \in Y$, but there are several equivalent ways to formalize this concept. In particular, the optimization of these ergodic averages is equivalent to the optimization of the integrals $\int f \mathrm{d} \nu$ over $S$-invariant Borel probability measures $\nu$ on $Y$. For the sake of brevity, we will call this the ``classical" theory of ergodic optimization. See \cite{JenkinsonSurvey} for a broad survey of this subject.

Here we introduce a generalization of ergodic optimization which we term ``adversarial ergodic optimization." We begin with a mostly informal summary of what adversarial ergodic optimization intends to do, then later give rigorous definitions that attempt to formalize this idea. The adversarial ergodic optimization can be summarized informally as the study of the quantity
\begin{align*}
	\addtocounter{equation}{1}\tag{\theequation}\label{Informal adversarial optimization} \sup_{y \in Y} \inf_{z \in \pi^{-1}\{y\}} \lim_{k \to \infty} \frac{1}{k} \sum_{j = 0}^{k - 1} f \left( U^j z \right) ,
\end{align*}
where $\pi : (Z, U) \to (Y, S)$ is a factor map of topological dynamical systems, and $f \in C(Z)$. It is sometimes useful to think about this adversarial ergodic optimization in terms of a game between two players, whom we name Alice and Bob. First, Alice picks a point $y \in Y$. Afterwards, Bob picks a point $z \in Z$ in the fiber of the point Alice has chosen, i.e. $z \in \pi^{-1} \{y\}$. This then determines a Birkhoff average, which -leaving aside the obvious issue of convergence- yields a limit $\lim_{k \to \infty} \frac{1}{k} \sum_{j = 0}^{k - 1} f \left( U^j z \right)$. Bob then pays Alice $\lim_{k \to \infty} \frac{1}{k} \sum_{j = 0}^{k - 1} f \left( U^j z \right)$ pesos. The expression (\ref{Informal adversarial optimization}) describes the optimal outcomes of this game. More generally, we understand this limit of Birkhoff averages in (\ref{Informal adversarial optimization}) as a placeholder for some context-appropriate quantity that quantifies the long-time dynamics of the function $f$ at the orbit $z$.

Adversarial ergodic optimization is interested in the ways in which the dynamics of $(Z, U)$ are ``constrained" by their image in $(Y, S)$ under $\pi$. In cases where $\pi$ is an isomorphism of topological dynamical systems, i.e. where $\pi$ preserves all information, this adversarial ergodic optimization reduces to the classical ergodic optimization. We adopt the name ``adversarial ergodic optimization" to emphasize the way that maximization and minimization interact with and antagonize each other.

The idea of an adversarial ergodic optimization has analogies in nature. Consider dynamical systems $(X, T), (Y, S)$, where $(X, T)$ describes the possible behaviors of a hare over time, and $(Y, S)$ describes possible behaviors of a fox, and consider the second coordinate projection $\pi : (X \times Y, T \times S) \to (Y, S)$. Imagine that a function $f : X \times Y \to \mathbb{R}$ describes the amount of ``access" the fox has to the hare as a source of food. The fox would like to maximize this quantity over time, since consistent access to its prey means a steady supply of food, while the hare would like to minimize this quantity. We could see this as a kind of adversarial ergodic optimization. Phenomena in this spirit, where the population dynamics of predators and prey affect each other, have been observed with cicadas. For instance, in \cite{CicadasWebb}, G.F. Webb proposed that the life cycles of certain cicada populations could be understood as a response to the long-term population dynamics of their predators, optimizing for the survival of the cicadas.

Adversarial ergodic optimization also arises in the symbolic-dynamical setting of \cite[Definition 5]{IteratedOptimization}, where the authors consider a quantity of the form
$$
\liminf_{k \to \infty} \max_{y \in Y} \min_{x \in X} \frac{1}{k} \sum_{j = 0}^{k - 1} f \left( T^j x , S^j y \right) .
$$
In their setting, both $(X, T)$ and $(Y, S)$ are shifts of finite type over a finite alphabet $\mathcal{A}$, and $f \in C(X \times Y)$ is the function $f(x, y) = 1 - \delta(x(0), y(0))$, where $\delta(\cdot, \cdot)$ here is the Kronecker delta $\chi_{=}(\cdot, \cdot)$. This value, which they call the ``covering radius" of the pair $(X, Y)$, measures (in an asymptotic sense) how far a word in $Y$ is from being a word in $X$, in the sense of how much a word must be ``perturbed" relative to its length. This parameter is relevant to describing the performance of an error-correction method they call ``quantized-constraint concatenation." While a fuller discussion of this error-correction scheme is beyond the scope of this article, this covering radius quantity resembles a functional of interest in the classical ergodic optimization theory that goes by various names, and when either of $(X, T), (Y, S)$ is a trivial topological dynamical system, this quantity reduces to that functional.

This adversarial ergodic optimization described informally by (\ref{Informal adversarial optimization}) can be formalized in several different ways. We focus here on four candidates for such a formalization, each of which generalizes one perspective on the classical ergodic optimization.

\begin{Not}
Given a factor map $\pi : (Z, U) \to (Y, S)$ of topological dynamical systems, as well as a continuous function $f \in C(Z)$, we consider the following functionals $C(Z) \to \mathbb{R}$:
\begin{align*}
	\alpha(f)	& = \sup_{\nu \in \mathcal{M}_S(Y)} \; \min_{\lambda \in \mathcal{M}_{U}(Z) \cap \pi_*^{-1} \{\nu\}} \int f \mathrm{d} \lambda , &
	\beta(f)	& = \sup_{ y \in \pi^{-1} R_f} \; \inf_{z \in R_f \cap \pi^{-1} \{y\}} \; \lim_{k \to \infty} \frac{1}{k} \sum_{j = 0}^{k - 1} f\left( U^j z \right) , \\
	\gamma(f)	& = \sup_{y \in Y} \; \inf_{z \in \pi^{-1} \{y\}} \; \limsup_{k \to \infty} \frac{1}{k} \sum_{j = 0}^{k - 1} f\left( U^j z \right) , &
	\delta(f)	& = \limsup_{k \to \infty} \; \max_{y \in Y} \; \min_{z \in \pi^{-1} \{y\}} \frac{1}{k} \sum_{j = 0}^{k - 1} f\left( U^j z \right) ,
\end{align*}
where $R_f : = \left\{ z \in Z : \textrm{$\lim_{k \to \infty} \frac{1}{k} \sum_{j = 0}^{k - 1} f\left( U^j z \right)$ exists} \right\}$, and $\mathcal{M}_U (Z)$ denotes the space of $U$-invariant Borel probability measures on $Z$ (c.f. Definition \ref{Def of TDS}).
\end{Not}

These quantities $\alpha, \beta, \gamma, \delta$ correspond to the optimization of integrals relative to invariant measures, optimization of ergodic averages over ``regular" orbits, optimization of limits superior of ergodic averages over all orbits, and ergodic averages of partial orbits of large length, respectively. In the context of classical ergodic optimization, these functionals are all equivalent, as can be seen in \cite[Section 2]{Jenkinson}. This equivalence can be seen as an expression of the heuristic in ergodic theory that pointwise ergodic averages are closely related to the invariant probability measures on a system (c.f. G.D. Birkhoff's pointwise ergodic theorem \cite{BirkhoffErgodicTheorem}).

The classical ergodic optimization can be recovered as a special case of our adversarial ergodic optimization in one of two ways. On one hand, we can consider the case where $(Z, U) = (Y, S)$, and $\pi : (Z, U) \to (Y, S)$ is an isomorphism of topological dynamical systems. On the other hand, we can consider the trivial factor map $\pi : (Z, U) \to (Y, S)$, i.e. the case where $(Y, S) = \left( \left\{ * \right\} , \operatorname{id}_{ \{*\} } \right)$ is the trivial topological dynamical system. In the former case, the adversarial ergodic optimization becomes a maximization over $(Y, S)$, and in the latter case, the adversarial ergodic optimization becomes a minimization over $(Z, U)$.

Our overarching goal in this article is to identify aspects of the classical ergodic optimization which generalize to the adversarial setting, as well as how the adversarial ergodic optimization can behave differently from the classical. In particular, we keep the following three questions in mind. Each question is followed by a discussion of current results surrounding them.

\begin{Qn}\label{Question: existence of maximizers}
Given a factor map of topological dynamical systems $\pi : (Z, U) \to (Y, S)$, and a continuous function $f \in C(Z)$, write $\psi_f(\nu) : = \min_{\lambda \in \mathcal{M}_U(Z) \cap \pi^{-1} \{\nu\}} \int f \mathrm{d} \lambda$. Under what conditions does there exist $\nu \in \mathcal{M}_S(Y)$ such that $\psi_f(\nu) = \alpha(f)$? Further, when such a maximizing measure exists, under what condition does there exist an ergodic maximizing measure?
\end{Qn}

Corollary \ref{Few measures on Y} and Theorem \ref{Theta map} provide sufficient conditions for the existence of an ergodic maximizing measure, though at present, we have no examples of a function that does not admit an ergodic maximizing measure. One could elaborate on Question \ref{Question: existence of maximizers} by asking when there exists a maximizing measure $\nu \in \mathcal{M}_S(Y)$ which satisfies some given property. A well-studied question along these lines in the classical ergodic optimization literature is when we can choose $\nu$ to be supported on a periodic orbit. This relates to the so-called ``Typically Periodic Optimization conjecture" (TPO conjecture) from the classical theory of ergodic optimization, a sort of meta-conjecture concerning conditions under which a ``typical" continuous function of sufficient regularity on a sufficiently hyperbolic topological dynamical system admits a maximizing measure with finite support. Our Proposition \ref{Periodic optimization} and Conjecture \ref{Periodic optimization conjecture} touch on the existence of maximizing measures with finite support. We refer the reader to \cite[Section 7]{JenkinsonSurvey} for a broad survey of the TPO conjecture and the literature surrounding it.

\begin{Qn}\label{Question: relations between functionals}
What relations between $\alpha(f), \beta(f), \gamma(f), \delta(f)$ hold for all factor maps of topological dynamical systems $\pi : (Z, U) \to (Y, S)$ and continuous functions $f \in C(Z)$?
\end{Qn}

At present, the only general inequality we know is that $\alpha(f) \leq \beta(f)$ (c.f. Theorem \ref{alpha leq beta}). While all four are known to be equivalent in the classical setting (c.f. \cite[Proposition 2.1]{Jenkinson}), Proposition \ref{Alpha delta counterexample} shows that some of these identities can fail. On the other hand, in the product-type case (i.e. $(Z, U) = (X, T) \times (Y, S), \pi : (x, y) \mapsto y$), Theorems \ref{Identity for alpha and delta under periodic specification} and \ref{Delta gamma inequality} show that if the underlying systems satisfy certain hyperbolicity conditions, then $\alpha = \gamma = \delta$.

\begin{Qn}\label{Question: delta limit existence}
Given a factor map of topological dynamical systems $\pi : (Z, U) \to (Y, S)$, and a continuous function $f \in C(Z)$, under what conditions does the limit $\lim_{k \to \infty} \max_{y \in Y} \min_{z \in \pi^{-1} \{y\}} \frac{1}{k} \sum_{j = 0}^{k - 1} f \left( U^j z \right)$ exist?
\end{Qn}

It's known that the limit exists in the classical setting as a consequence of Fekete's lemma. We show in Proposition \ref{Delta limit exists for spec property} that the limit exists when $(X, T)$ has the specification property, and \cite[Proposition 7]{IteratedOptimization} provides other sufficient conditions for this limit to exist. At present, we have no example where the limit does not exist.

For the purposes of this article, we will mostly focus on the product-type case, where $\pi : (X \times Y, T \times S) \to (Y, S)$ is the projection of a product system onto the second coordinate. Section \ref{General theory section} is the only section not to focus on this product case. For reasons we will elaborate more on later in this article, a fuller study of this more general setting of factors will be the subject of a future work.

This article is organized into sections as follows. For each section, we include a brief description of the main results.

In Section \ref{General theory section}, we establish that $\alpha \leq \beta$ holds in general, without any further assumptions on the underlying topological dynamical systems. This is the only inequality of its kind in this article. We also provide some sufficient conditions for the existence of an ``optimal" measure exists for $\alpha$, i.e. a measure which attains the supremum. This section also sets some of the notation and conventions that we will use throughout this article, and includes a few elementary lemmas which will be used throughout the article. This is the only section that focuses on the case where $\pi : (Z, U) \to (Y, S)$ is a general factor map of topological dynamical systems, as later sections focus on the special case of ``product-type dynamics," where $(Z, U) = (X, T) \times (Y, S)$ for some topological dynamical systems $(X, T), (Y, S)$, and $\pi : (x, y) \mapsto y$ is the projection onto the second coordinate. This section also includes an example showing that in general, the functionals $\alpha, \beta, \gamma, \delta$ need not be equal, demonstrating that some kind of additional restrictions on the underling systems are needed for these quantities to be equivalent. Our main results in this section are:
\begin{itemize}
	\item Theorem \ref{alpha leq beta}: $\alpha(f) \leq \beta(f)$ for all systems.
	\item Proposition \ref{Alpha delta counterexample}: Unlike in the classical setting, $\alpha(f) = \beta(f) = \gamma(f) = \delta(f)$ does not hold in general.
\end{itemize}

In Section \ref{Specification section}, we establish some properties of and relations between the functionals $\alpha, \beta, \gamma, \delta$ under the assumption that the underlying systems satisfy certain hyperbolicity conditions, namely specification properties. In particular, we show that $\alpha = \gamma = \delta$ under these conditions. Our main results of this section are as follows:
\begin{itemize}
	\item Proposition \ref{Delta limit exists for spec property}: If $(X, T)$ has the specification property, then the limit superior defining $\delta(f)$ can be replaced with a limit.
	\item Theorem \ref{Identity for alpha and delta under periodic specification}: If $(X, T), (Y, S)$ have the periodic specification property, then $\alpha(f) = \delta(f)$.
	\item Theorem \ref{Delta gamma inequality}: If $(X, T), (Y, S)$ have the specification property, then $\gamma(f) = \delta(f)$.
\end{itemize}

In Section \ref{Locally constant secction}, we take up a question raised in Section \ref{General theory section} concerning the existence of invariant measures which are in some sense optimal in the context of adversarial ergodic optimization. We study this question specifically in the context of transitive shifts of finite type. We find that when the function in question is locally constant, the optimizing measures can be characterized in terms of certain subshifts made of generalized ground-state configurations. Consequently, questions about these optimizing measures are equivalent to questions about the subshifts associated to these functions. We consider this setting to be of particular importance, since our study of adversarial ergodic optimization is inspired by earlier work on the optimization of locally constant functions on shifts of finite type in \cite{IteratedOptimization}. Our main results in this section are as follows:
\begin{itemize}
	\item Propositions \ref{Measures supported on ground states are optimizing}, \ref{Optimizing measures supported on ground states}: Consider shifts of finite type $(X, T), (Y, S)$, and suppose that $f \in C(X \times Y)$ is locally constant. Then there exists a shift $Y_{f, 0} \subseteq Y$ with the property that a measure $\nu_0 \in \mathcal{M}_S(Y)$ is supported on $Y_{f, 0}$ if and only if $\nu_0$ is ``$\alpha$-maximizing" in the sense that
	$\alpha(f) = \min_{\lambda \in \mathcal{M}_{T \times S} (X \times Y) \cap (\pi_Y)_*^{-1} \{\nu_0\}} \int f \mathrm{d} \lambda$ (c.f. Definition \ref{Notation for psi}).
	\item Theorem \ref{Theta map}: If $(X, T), (Y, S)$ are shifts, and $(X, T)$ is a transitive shift of finite type, then for every $f \in C(X \times Y)$ exists an $\alpha$-maximizing $\nu_0 \in \mathcal{M}_S(Y)$. In particular, the shift $Y_{f, 0}$ described above is nonempty.
	\item Corollary \ref{Subordination principle}: A subordination principle in the sense of T. Bousch (c.f. \cite{BouschWalters}) holds for $\alpha$-maximizing measures for locally constant functions on a product of transitive shifts of finite type, where any measure whose support is contained in the support of an $\alpha$-maximizing measure is itself $\alpha$-maximizing.
\end{itemize}

We thank Tom Meyerovitch for his substantial contributions to this work. These were communicated to us via many helpful discussions during the period in which this work was carried out. We have also made a number of changes to this paper in light of an anonymous referee's very helpful comments. This research was supported by Israel Science Foundation grant no. 985/23.

\section{Defining the adversarial ergodic optimization}\label{General theory section}

\begin{Conv}
Throughout this article, all scalar-valued functions are assumed to take values in $\mathbb{R}$, unless otherwise specified. In particular, for a compact metrizable space $X$, the notation $C(X)$ will always refer to the real Banach space of continuous functions $X \to \mathbb{R}$ equipped with the uniform norm.
\end{Conv}

\begin{Def}\label{Def of TDS}
A \emph{topological dynamical system} is a pair $(X, T)$ consisting of a compact metrizable space $X$ and a homeomorphism $T : X \to X$. Given a topological dynamical system $(X, T)$, we define $\mathcal{M}_T(X)$ to be the space of all $T$-invariant Borel probability measures on $X$ endowed with the weak* topology.
\end{Def}

\begin{Not}\label{Def of average notation}
Given a topological dynamical system $(Z, U)$, a continuous function $f \in C(Z)$, a Borel measure $\lambda$ on $Z$, and a nonempty finite set $J \subseteq \mathbb{Z}$, we write
\begin{align*}
\mathbb{A}_J f	& : = \frac{1}{|J|} \sum_{j \in J} f \circ U^j ,	& \mathbb{S}_J f	& : = \sum_{j \in J} f \circ U^j \\
\mathbb{A}_J \lambda	& : = \frac{1}{|J|} \sum_{j \in J} U_*^j \lambda,	& \mathbb{S}_J \lambda	& : = \sum_{j \in J} U_*^j \lambda .
\end{align*}
In particular, for $k \in \mathbb{N}$, we set $\mathbb{A}_k : = \mathbb{A}_{[0, k - 1]}, \mathbb{S}_k : = \mathbb{S}_{[0, k - 1]}$. While we will sometimes use the $\mathbb{A}, \mathbb{S}$ notations with respect to underlying topological dynamical systems not called $(Z, U)$, we will only every use this notation in contexts where it is clear what the underlying system is.
\end{Not}

We define our topological dynamical systems to be invertible, but most results do not require this assumption, and would hold just as well if we only assumed that the transformation was continuous and surjective. Our decision to consider only invertible topological dynamical systems serves primarily to streamline certain definitions and arguments. We also prefer to speak in terms of metrizable spaces rather than metric spaces, because we will sometimes want the freedom to choose a metric for a space that is convenient for a specific purpose.

For the remainder of this section, let $\pi : (Z, U) \to (Y, S)$ be a factor map of topological dynamical systems, i.e. a surjective map $\pi : Z \to Y$ such that $S \circ \pi = \pi \circ U$, where $(Y, S), (Z, U)$ are topological dynamical systems.

The functionals $\alpha, \beta, \gamma, \delta$ are based on the functionals of the same name presented in \cite[Section 2]{Jenkinson}, and when we take $\pi$ to be trivial, i.e. $(Y, S) = (Z, U), \pi = \operatorname{id}_Z$, our functionals here are equivalent those of the same name in \cite{Jenkinson}. All four functionals describe in some sense the kind of extreme behaviors our adversarial ergodic optimization is interested in. Notably, in the classical setting, we have the following for all $f \in C(Z)$:
\begin{enumerate}
	\item The supremum in our definition of $\alpha(f)$ is in fact a maximum,
	\item the limit superior in our definition of $\delta(f)$ is in fact a limit, and
	\item $\alpha(f) = \beta(f) = \gamma(f) = \delta(f)$.
\end{enumerate}
In the classical setting, claims (1) and (3) can be found in \cite{Jenkinson}. Claim (2) is a direct consequence of the sequence $k \mapsto \max_{y \in Y} \sum_{j = 0}^{k - 1} f \left( S^j y \right) $ being subadditive, meaning that Fekete's lemma \cite{Fekete} applies. In light of claim (3), some authors refer to the value $\alpha(f) = \beta(f) = \gamma(f) = \delta(f)$ by names such as the ``ergodic maximizing value," as in \cite{GaribaldiErgodicOptimization}.

In this section and Section \ref{Specification section}, we will attempt to generalize these aforementioned facts about $\alpha, \beta, \gamma, \delta$ to our adversarial optimization setting, as well as note how they fail to generalize.

\begin{Rmk}
Our definition of $\delta(f)$ might remind the reader of the subadditive ergodic optimization as considered in (to name only one example) \cite[Appendix A]{MorrisMather}, where one considers a subadditive sequence of (upper semi-)continuous functions $(\varphi_k)_{k = 1}^\infty$ on a compact metrizable space $Y$, and looks at $\sup_{y \in Y} \lim_{k \to \infty} \varphi_k(y) / k$. In our adversarial ergodic optimization, the sequence $\left( y \mapsto \min_{z \in \pi^{-1} \{y\}} \sum_{j = 0}^{k - 1} f \left( U^j z \right) \right)_{k = 1}^\infty$ is a \emph{super}additive sequence of continuous functions, so one might imagine that the theory of subadditive (or equivalently, superadditive) ergodic optimization has something to say here. There is an important difference, however, between adversarial ergodic optimization and sub/superadditive ergodic optimization. In the ``classical" ergodic optimization, the choice of whether to maximize or minimize the ergodic averages of a continuous function is arbitrary. However, the minimization of the average $\varphi_k / k$ for a superadditive sequence $(\varphi_k)_{k = 1}^\infty$, which the ``subadditive ergodic optimization" is built around, is meaningfully different from the maximization of that same sequence, which our adversarial ergodic optimization is interested in.
\end{Rmk}

\begin{Def}\label{Notation for psi}
Consider a factor map $\pi : (Z, U) \to (Y, S)$ of topological dynamical systems. For $f \in C(Z) , \nu \in \mathcal{M}_S(Y)$, set
	$$\psi_f(\nu) : = \min \left\{ \int f \mathrm{d} \lambda : \lambda \in \mathcal{M}_{U}(Z), \; \pi_* \lambda = \nu \right\} . $$
	When $\psi_f$ attains a supremum at $\nu_0 \in \mathcal{M}_S(Y)$, we call $\nu_0$ an \emph{$\alpha$-maximizing measure} for $f$.

Given a convex set $\mathcal{K}$, we denote by $\partial_e \mathcal{K}$ the set of extremal points of $\mathcal{K}$, i.e. the set of all points $\lambda \in \mathcal{K}$ for which if $\kappa_1, \kappa_2 \in \mathcal{K}$, and $\frac{\kappa_1 + \kappa_2}{2} = \lambda$, then $\kappa_1 = \kappa_2 = \lambda$.
\end{Def}

\begin{Lem}\label{Ergodic points in preimage of marginal}
Let $\pi : (Z, U) \to (Y, S)$ be a factor map of topological dynamical systems, and let $\nu \in \partial_e \mathcal{M}_S(Y)$ be ergodic. Set $\mathcal{K} = \left\{ \lambda \in \mathcal{M}_{U}(Z) : \pi_* \lambda = \nu \right\}$. Then $\mathcal{K}$ is compact and convex in the weak* topology, and moreover the extreme points of $\mathcal{K}$ are exactly the ergodic elements of $\mathcal{K}$.
\end{Lem}

\begin{proof}
That $\mathcal{K}$ is compact and convex is immediate from the fact that $\pi_* : \mathcal{M}_{U}(Z) \to \mathcal{M}_S(Y)$ is continuous and affine. Now we argue that the extreme points of $\mathcal{K}$ are exactly the ergodic elements of $\mathcal{K}$, i.e. $\partial_e \mathcal{K} = \mathcal{K} \cap \partial_e \mathcal{M}_U(Z)$, which we prove by a standard double containment argument.
\begin{subproof}[Proof of $\partial_e \mathcal{K} \subseteq \mathcal{K} \cap \partial_e \mathcal{M}_U(Z)$]
We prove the contrapositive. Suppose that $\lambda \in \mathcal{K}$ is \emph{not} ergodic. Then there exist $\kappa, \theta \in \mathcal{M}_{U}(Z), \kappa \neq \theta$ such that $\lambda = \frac{\kappa + \theta}{2}$. Thus $\pi_* \lambda = \frac{\pi_* \kappa + \pi_* \theta}{2}$. Since $\nu \in \partial_e \mathcal{M}_S(Y)$ is ergodic, this implies that $\pi_* \kappa = \pi_* \theta = \nu$, meaning that $\kappa, \theta \in \mathcal{K}$. Therefore $\lambda = \frac{\kappa + \theta}{2}$ is not an extremal point of $\mathcal{K}$.
\end{subproof}
\begin{subproof}[Proof of $\partial_e \mathcal{K} \supseteq \mathcal{K} \cap \partial_e \mathcal{M}_U(Z)$]
If $\lambda \in \mathcal{K}$ is ergodic, then it must be extremal in $\mathcal{K}$, since if $\lambda$ was a nontrivial convex combination of points in $\mathcal{K}$, then a fortiori, $\lambda$ would also be a nontrivial convex combination of points in $\mathcal{M}_{U}(Z)$, contradicting the ergodicity of $\lambda$.
\end{subproof}
\end{proof}

\begin{Lem}\label{Lower semi-continuity}
The functional $\psi_f (\cdot) : \mathcal{M}_S(Y) \to \mathbb{R}$ is lower semi-continuous and convex for fixed $f \in C({Z})$.
\end{Lem}

\begin{proof}
Fix $\nu \in \mathcal{M}_S(Y)$, and let $(\nu_k)_{k = 1}^\infty$ be a sequence in $\mathcal{M}_S(Y)$ converging to $\nu$. For each $k \in \mathbb{N}$, choose $\lambda_k \in \pi_*^{-1} \{\nu_k\}$ such that $\int f \mathrm{d} \lambda_k = \psi_f (\nu_k)$.
	
In order to prove that $\liminf_{k \to \infty} \psi_f(\nu_k) \geq \psi_f(\nu)$, it suffices to prove that for any $k_1 < k_2 < \cdots$ such that the sequence $\left( \psi_f \left( \nu_{k_\ell} \right) \right)_{\ell = 1}^\infty$ converges, we have $\lim_{\ell \to \infty} \psi_f \left( \nu_{k_\ell} \right) \geq \psi_f(\nu)$, so suppose we have such a sequence $(k_\ell)_{\ell = 1}^\infty$. By passing to a further subsequence if necessary, we can assume that $\left( \lambda_{k_\ell} \right)_{\ell = 1}^\infty$ converges to a limit, say $\lim_{\ell \to \infty} \lambda_{k_\ell} = \theta$. Due to the continuity of $(\pi_Y)_*$, we can infer that $\pi_* \theta = \nu$. Therefore
$$ \psi_f(\nu) \leq \int f \mathrm{d} \theta = \lim_{\ell \to \infty} \int f \mathrm{d} \lambda_{k_\ell} = \lim_{\ell \to \infty} \psi_f \left( \nu_{k_\ell} \right) . $$
Therefore $\psi_f$ is lower semi-continuous.
	
Next, we show that $\psi_f$ is convex. Let $\omega_1, \omega_2 \in \mathcal{M}_S(Y); t \in [0, 1]$, and choose $\kappa_i \in \pi_*^{-1} \{\nu_i\}$ such that $\psi_f(\omega_i) = \int f \mathrm{d} \kappa_i$ for $i = 1, 2$. Then
$$\psi_f(t \omega_1 + (1 - t) \omega_2) \leq \int f \mathrm{d} (t \kappa_1 + (1 - t) \kappa_2) = t \psi_f(\omega_1) + (1 - t) \psi_f(\omega_2) .$$
Thus $\psi_f$ is convex.
\end{proof}

We can use this $\psi_f$ notation to write the definition of $\alpha$ more concisely as
$$\alpha(f) = \sup_{\nu \in \mathcal{M}_S(Y)} \psi_f (\nu) ,$$
but as the following lemma demonstrates, this supremum can be further restricted to ergodic measures.

\begin{Lem}\label{Sup on ergodic measures}
Consider a factor map $\pi : (Z, U) \to (Y, S)$ of topological dynamical systems, and a continuous function $f \in C({Z})$. Then
\begin{align*}
	\alpha(f) = \sup_{\nu \in \partial_e \mathcal{M}_S(Y)} \psi_f(\nu) = \sup_{\nu \in \partial_e \mathcal{M}_S(Y)} \left( \inf_{\lambda \in \left( \partial_e \mathcal{M}_{U}(Z) \right) \cap \pi_*^{-1} \{\nu\}} \left( \int f \mathrm{d} \lambda \right) \right) .
\end{align*}
\end{Lem}

\begin{proof}
This follows from Lemmas \ref{Ergodic points in preimage of marginal} and \ref{Lower semi-continuity}.
\end{proof}

Question \ref{Question: existence of maximizers} asks when $\psi_f : \mathcal{M}_S(Y) \to \mathbb{R}$ attains a maximum. There are two ``trivial" cases in which this happens. On one hand, if $\pi : (Z, U) \to (Y, S)$ is an isomorphism of topological dynamical systems, then $\psi_f(\nu) = \int f \circ \pi^{-1} \mathrm{d} \nu$, and so $\psi_f$ attains a maximum at an ergodic measure (c.f. \cite[Proposition 2.4]{Jenkinson}). On the other hand, when $(Y, S)$ is a trivial dynamical system, the space $\mathcal{M}_S(Y)$ is trivial, and therefore $\psi_f$ attains a maximum at an ergodic measure, i.e. the only point in its domain. In other words, our trivial cases are where $\pi$ is an isomorphism, and thus preserves all information about $(Z, U)$; and where $\pi$ is a trivial map, and thus preserves no information about $(Z, U)$. Corollary \ref{Few measures on Y} can be seen as an approximation of the latter case, where $\mathcal{M}_S(Y)$ only hosts relatively few invariant measures, and so in that sense, the system $(Y, S)$ is ``close to trivial."

\begin{Cor}\label{Few measures on Y}
Consider a factor map $\pi : (Z, U) \to (Y, S)$ of topological dynamical systems, and a continuous function $f \in C(Z)$. If $\partial_e \mathcal{M}_S(Y)$ contains only finitely many points, then there exists an ergodic $\alpha$-maximizing $\nu \in \partial_e \mathcal{M}_S(Y)$.
\end{Cor}

\begin{proof}
This is a direct consequence of Lemma \ref{Sup on ergodic measures}, since when $\partial_e \mathcal{M}_S(Y)$ is finite, the supremum in $\alpha(f) = \sup_{\nu \in \partial_e \mathcal{M}_S(Y)} \psi_f(\nu)$ is a maximum.
\end{proof}

Section \ref{Locally constant secction} returns to the study of measures $\nu \in \mathcal{M}_S(Y)$ such that $\psi_f(\nu) = \alpha(f)$ in the context of locally constant functions $f \in C(X \times Y)$ when $(X, T), (Y, S)$ are transitive shifts of finite type, as well as what properties we can expect of such measures. In particular, Theorem \ref{Theta map} shows that for any $f \in C(X \times Y)$, the functional $\psi_f : \mathcal{M}_S(Y) \to \mathbb{R}$ is continuous and affine, and thus has an ergodic maximizer.

\begin{Thm}\label{alpha leq beta}
Consider a factor map $\pi : (Z, U) \to (Y, S)$ of topological dynamical systems, and $f \in C(Z)$. Assume further that $\pi$ is open. Then $\alpha(f) \leq \beta(f)$.
\end{Thm}

\begin{proof}
Fix $\nu \in \partial_e \mathcal{M}_S(Y)$, and choose $\lambda \in \partial_e \mathcal{M}_{U}(Z)$ such that $\pi_* \lambda = \nu$. Then there exists $z_\lambda \in Z, y_\nu \in Y$ such that $\mathbb{A}_k \delta_{z_\lambda} \stackrel{k \to \infty}{\to} \lambda$, and $\pi z_\lambda = y_\nu$. Then
		\begin{align*}
			\inf_{z \in R_f^{-1} \{y_\nu\}} \lim_{k \to \infty} \mathbb{A}_k f \left( z \right)	& \leq \inf_{ \lambda \in \left( \partial_e \mathcal{M}_{U}(Z) \right) \cap \pi_*^{-1} \{\nu\} } \int f \mathrm{d} \lambda = \psi_f(\nu) .
		\end{align*}
		Now we prove the opposite inequality. Fix $z \in R_f^{-1} \{y_\nu\}$. Then the sequence $\left( \mathbb{A}_k \delta_{z } \right)_{k = 1}^\infty$ admits a subsequence converging to some $\lambda' \in \mathcal{M}_{U}(Z)$, and $\int f \mathrm{d} \lambda ' = \lim_{k \to \infty} \mathbb{A}_k f \left( z \right)$. Thus
		\begin{align*}
			\lim_{k \to \infty} \mathbb{A}_k f \left( z \right) = \int f \mathrm{d} \lambda '	& \geq \psi_f(\nu) .
		\end{align*}
		Therefore, for every $\nu \in \partial_e \mathcal{M}_{S}(Y)$ exists $y_\nu \in Y$ such that
		\begin{align*}
			\psi_f(\nu) = \inf_{R_f^{-1} \{y_\nu\}} \lim_{k \to \infty} \mathbb{A}_k f \left( z \right) \leq \beta(f) .
		\end{align*}
		Taking the supremum over $\nu \in \partial_e \mathcal{M}_S(Y)$ and appealing to Lemma \ref{Sup on ergodic measures} tells us that $\alpha(f) \leq \beta(f)$.
\end{proof}

As of right now, Theorem \ref{alpha leq beta} is the only known inequality between any of the quantities $\alpha, \beta, \gamma, \delta$ that holds for \emph{all} topological dynamical systems. However, in Section \ref{Specification section}, we will investigate some other inequalities and identities that can hold when we restrict to suitable systems.

\begin{Prop}\label{Alpha delta counterexample}
There exists a factor map of topological dynamical systems $\pi : (X, T) \times (Y, S) \to (Y, S)$ for which there exists $f \in C(X \times Y)$ such that $\alpha(f) = \beta(f) = \gamma(f) < \delta(f)$. Furthermore, if we denote by $\beta^*, \gamma^*, \delta^*$ the analogous functionals corresponding to the ``two-sided" ergodic averages $\frac{1}{2k + 1} \sum_{j = - k}^k (T \times S)^j$, then we can have that $\alpha < \beta^* = \gamma^* = \delta^* = \delta$.
\end{Prop}

\begin{proof}
	First, we define the system $X = \{- 1, + 1\}, T = \operatorname{id}_X$. Let $Y$ be the ``two-point compactification" of $\mathbb{Z}$, i.e. $Y = \mathbb{Z} \cup \left\{ - \infty, \infty \right\}$, endowed with the topology that has the basis $\left\{ \{n\} , [- \infty, n), (n, + \infty] : n \in \mathbb{Z} \right\}$. Let $S : Y \to Y$ be the homeomorphism
	\begin{align*}
		S y	& = \begin{cases}
			y + 1	& \textrm{if $y \in \mathbb{Z}$}, \\
			y	& \textrm{if $y = \pm \infty$} .
		\end{cases}
	\end{align*}
Finally, let $f \in C(X \times Y)$ be the function $f(x, y) = - \operatorname{sgn}(x y)$, i.e.
	\begin{align*}
		f(x, y)	& = \begin{cases}
			- 1	& \textrm{if $x y \geq 1$} , \\
			1	& \textrm{if $x y \leq - 1$} , \\
			0	& \textrm{if $y = 0$} .
		\end{cases}
	\end{align*}
	Over the course of this proof, we're going to show that $\alpha(f) = \beta(f) = \gamma(f) = - 1$, and $\delta(f) = 0$.
	
	By Lemma \ref{Sup on ergodic measures}, when calculating $\alpha(f)$, it suffices to optimize over ergodic measures $\nu \in \partial_e \mathcal{M}_S(Y)$. We can see quickly that
	\begin{align*}
		\partial_e \mathcal{M}_T(X)	& = \left\{ \delta_{- 1}, \delta_{1} \right\} , &
		\partial_e \mathcal{M}_S(Y)	& = \left\{ \delta_{- \infty}, \delta_{+ \infty} \right\} , &
		\partial_e \mathcal{M}_{T \times S}(X \times Y)	& = \left( \partial_e \mathcal{M}_T(X) \right) \times \left( \partial_e \mathcal{M}_S(Y) \right) .
	\end{align*}
	It follows that $\alpha(f) = - 1$. To compute $\beta(f), \gamma(f)$, we can observe that the limit $\lim_{k \to \infty} \left( T^k x, S^k y \right)$ exists for every $(x, y) \in X \times Y$, meaning that $R_f = X \times Y$, implying that $\beta(f) = \gamma(f)$. Let $F : X \times Y \to \{-1, 1\}$ be the function
	\begin{align*}
		F(x, y)	& =  \lim_{k \to \infty} f \left( T^k x , S^k y \right) = \begin{cases}
			- 1	& \textrm{if $x = - 1, y = - \infty$} , \\
			1	& \textrm{if $x = 1, y = - \infty$} , \\
			1	& \textrm{if $x = - 1, y \neq - \infty$} , \\
			- 1	& \textrm{if $x = 1, y \neq - \infty$} .
		\end{cases}
	\end{align*}
	For $y \in Y$, we have that
	\begin{align*}
	\inf_{x \in X} \lim_{k \to \infty} \mathbb{A}_k f \left( x, y \right)	& = \inf_{x \in X} F(x, y) = - 1 .
	\end{align*}
	Thus $\beta(f) = \gamma(f) = - 1$.
	
	Now we compute $\delta(f)$. Fix $k \in \mathbb{N}$. We see that
	\begin{align*}
		\mathbb{A}_k f \left( \pm 1 , + \infty \right)	& = \mp 1 ,	& \mathbb{A}_k f \left( \pm 1 , - \infty \right)	& = \pm 1 .
	\end{align*}
	Thus in particular $\min_{x \in X} \mathbb{A}_k f \left( x , \pm \infty \right) = - 1$. Consider now the case where $y \in \mathbb{Z}$. Then we have
	\begin{align*}
		\mathbb{A}_k f \left( x , y \right)	& = \begin{cases}
			\frac{\left| [y, y + k - 1] \cap (- \infty, - 1] \right| - \left| [y, y + k - 1] \cap [1, + \infty) \right|}{k}	& \textrm{if $x = + 1$,}	\\
			\frac{\left| [y, y + k - 1] \cap [1, + \infty) \right| - \left| [y, y + k - 1] \cap (- \infty, - 1] \right|}{k}	& \textrm{if $x = - 1$,}
		\end{cases}
	\end{align*}
	where $|E|$ denotes the cardinality of $E$ for a finite set $E$. Thus
	\begin{align*}
		\min_{x \in X} \mathbb{A}_k f \left( x , y \right)	& = - \frac{1}{k} \left| \left| [y, y + k - 1] \cap [1, + \infty) \right| - \left| [y, y + k - 1] \cap (- \infty, -1] \right| \right| \leq 0 .
	\end{align*}
	We claim, however, that
	\begin{align*}
		- \frac{1}{k} \leq \max_{y \in Y} \min_{x \in X} \mathbb{A}_k f \left( x , y \right) \leq 0
	\end{align*}
	for all $k \in \mathbb{N}$. The upper bound has already been demonstrated, and the lower bound can be established by looking at $y = - \lfloor k / 2 \rfloor$. Therefore $\lim_{k \to \infty} \max_{y \in Y} \min_{x \in X} \mathbb{A}_k f(x, y)$ exists and is equal to $0$. In particular, this means that $\delta(f) = 0 > - 1$.

Consider now the set
\begin{align*}
	R_f^*	& = \left\{ z \in Z : \textrm{$\lim_{k \to \infty} \mathbb{A}_{[-k, k]} f(z)$ exists} \right\} ,
\end{align*}
along with the following functionals $C(Z) \to \mathbb{R}$:
\begin{align*}
	\beta^*(f)	& = \sup_{ y \in \pi^{-1} R_f^*} \; \inf_{z \in R_f \cap \pi^{-1} \{y\}} \; \lim_{k \to \infty} \mathbb{A}_{[-k, k]} f(z) , \\
	\gamma^*(f)	& = \sup_{y \in Y} \; \inf_{z \in \pi^{-1} \{y\}} \; \limsup_{k \to \infty} \mathbb{A}_{[-k, k]} f(z) , \\
	\delta^*(f)	& = \limsup_{k \to \infty} \; \max_{y \in Y} \; \min_{z \in \pi^{-1} \{y\}} \mathbb{A}_{[-k, k]} f(z) .
\end{align*}
Then $R_f^* = X \times Y$, and
\begin{align*}
	\lim_{k \to \infty} \mathbb{A}_{[-k, k]} f(x, y)	& = \begin{cases}
		0	& \textrm{if $y \in \mathbb{Z}$} , \\
		-1	& \textrm{if $x = \pm 1, y = \pm \infty$,} \\
		1	& \textrm{if $x = \pm 1, y = \mp \infty$.}
	\end{cases}
\end{align*}
In particular, we see that
\begin{align*}
	\sup_{y \in Y} \inf_{x \in X} \limsup_{k \to \infty} \mathbb{A}_{[-k, k]} f(x, y)	& = 0 .
\end{align*}
We can also observe that
\begin{align*}
	0 = \min_{x \in X} \mathbb{A}_{[-k, k]} f(x, y) \leq \max_{y \in Y} \min_{x \in X} \mathbb{A}_{[-k, k]} f(x, y) \leq 0 .
\end{align*}
Thus $\alpha(f) = - 1 < \beta^*(f) = \gamma^*(f) = \delta^*(f) = \delta(f) = 0$.
\end{proof}

The fact that $\alpha, \beta, \beta^*, \gamma, \gamma^*, \delta, \delta^*$ all agree in the ``classical" ergodic optimization setting (see the discussion in Section \ref{General theory section}) is a manifestation of a general heuristic in egodic theory, that the limiting behaviors of ergodic averages and the invariant probability measures of a system are ``two sides of the same coin," perhaps best exemplified by the classical ergodic theorems of G.D. Birkhoff and J. von Neumann (c.f. \cite{BirkhoffErgodicTheorem, VNErgodicTheorem}). In particular, this classical equivalence of $\alpha, \beta, \gamma, \delta$ asserts that the extreme behaviors of a continuous function's ergodic averages over long time (a concept which $\beta(f), \gamma(f), \delta(f)$ present different interpretations of) are the same as the extreme behaviors of that function's integral under invariant Borel probability measures (interpreted as $\alpha$). The fact that these functionals can be distinct in our setting of adversarial ergodic optimization deviates from this heuristic, as seen in Proposition \ref{Alpha delta counterexample}. Moreover, the fact that $\beta = \gamma \neq \delta$ indicates that the limiting behaviors of individual orbits (which $\beta, \gamma$ quantify) don't always agree with the behaviors of very long orbit segments (which $\delta$ quantifies), another deviation from the classical theory.

Proposition \ref{Alpha delta counterexample} demonstrates that in order to have $\alpha(f) = \delta(f)$ for all $f \in C(X \times Y)$, some sort of additional assumption is needed on $(X, T), (Y, S)$. This is the motivation for Section \ref{Specification section}.

\section{Relations between $\alpha, \beta, \gamma, \delta$ under specification conditions}\label{Specification section}

In this section, we investigate the adversarial ergodic optimization in systems with certain specification properties. The study of specification properties began with R. Bowen's work on Axiom A diffeomorphisms (c.f. \cite{BowenSpec}). Today, many different specification properties have been identified and studied. While we define all the relevant properties here, the interested reader can consult \cite{Panorama} for a broad overview of these different specification properties and their relationships to each other.

\begin{Def}
Let $(X, T)$ be a topological dynamical system, and let $d$ be a compatible metric on $X$. A \emph{specification in $X$} is a finite sequence $\xi = \left( x_i, [a_i, b_i] \right)_{i = 1}^n$, where $x_i \in X$, and $a_i, b_i$ are integers such that $a_1 \leq b_1 < a_2 \leq b_2 < \cdots < a_n \leq b_n$. We call that specification \emph{$D$-spaced} for some $D \in \mathbb{N}$ if $a_{i + 1} - b_i \geq D$ for all $i = 1, \ldots, n - 1$. Given a specification $\xi = \left( x_i, [a_i, b_i] \right)_{i = 1}^n$ and a positive $\eta > 0$, we call a point $x \in X$ an \emph{$\eta$-tracing} of $\xi$ if
\begin{align*}
d \left( T^j x , T^j x_i \right)	& < \eta	& \textrm{for all $i \in [1, n], j \in [a_i, b_i]$.}
\end{align*}
	
We say a topological dynamical system $(X, T)$ with compatible metric $d$ has the \emph{periodic specification property} if for every $\eta > 0$ exists a natural number $D = D(\eta) \in \mathbb{N}$ such that if $\xi = \left( x_i, [a_i, b_i] \right)_{i = 1}^n$ is a $D$-spaced specification, then there exists $x' \in X$ such that $x'$ is an $\eta$-tracing of $\xi$ and $T^{b_n - a_1 + D} x' = x'$. If we omit the requirement that $x'$ be periodic, then we say simply that $(X, T)$ has the \emph{specification property}.
\end{Def}

The importance of \emph{periodic} specification for our purposes is that if $T^t x = x$ for some $t \in \mathbb{N}, x \in X$, then this periodicity information about $x$ not only tells us that the sequence $\left( \mathbb{A}_k \delta_{x} \right)_{k = 1}^\infty$ converges, but gives us an estimate of the rate at which it converges. Note that for shifts of finite type (defined in Definition \ref{Definition of SFT}), the periodic specification property is equivalent to a mixing condition or transitivity (c.f. \cite[Theorem 47 and associated references]{Panorama}).

\begin{Rmk}
The terminology for different specification properties is not universal. For example, some authors use the term ``specification property" to refer to what we call here the periodic specification property, and use ``weak specification property" to refer to what we call the specification property. There are similar differences in terminology for other specification properties. The interested reader can consult \cite{Panorama} for more discussion of the different terminologies that appear in the literature.
\end{Rmk}

\begin{Not}
Throughout this section, we use $d_X, d_Y$ to denote compatible metrics on $X, Y$, respectively. We always take $d_{X \times Y}$ to be the metric on $X \times Y$ given by
$$d_{X \times Y}( (x_1, y_1) , (x_2, y_2) ) : = \max\left\{ d_X(x_1, x_2), d_Y(y_1, y_2) \right\} .$$

For a topological dynamical system $(X, T)$ and a natural number $t \in \mathbb{N}$, we write
\begin{align*}
\operatorname{Per}_T^t(X)	& : = \left\{ x \in X : T^t x = x \right\} ,	& \operatorname{Per}_T(X)	& : = \bigcup_{t = 1}^\infty \operatorname{Per}_T^t(X) .
\end{align*}
We write $\mathcal{M}_T^{co}(X)$ for the set of $T$-invariant probability measures on $X$ supported on a periodic orbit, i.e.
$$\mathcal{M}_T^{co}(X) : = \left\{ \mathbb{A}_k \delta_{x} \in \mathcal{M}_T(X) : k \in \mathbb{N}, x \in \operatorname{Per}_T^k(X) \right\} .$$
\end{Not}

\begin{Rmk}
The superscript $co$ stands for ``closed orbit," as explained in \cite[(21.7)]{DGSonCompactSpaces}.
\end{Rmk}

We now introduce periodic variations of the $\alpha, \beta, \gamma, \delta$ functionals.

\begin{Not}
Let $(X, T), (Y, S)$ be topological dynamical systems, and consider $f \in C(X \times Y)$. We define the following quantities:
\begin{align*}
\alpha_{per}(f)	& = \sup_{\nu \in \mathcal{M}_S^{co}(Y)} \inf_{\lambda \in \mathcal{M}_{T \times S}^{co}(X \times Y) \cap (\pi_Y)_*^{-1} \{\nu\}} \int f \mathrm{d} \lambda , \\
\beta_{per}(f)	& = \sup_{y \in \operatorname{Per}_S(Y)} \inf_{x \in \operatorname{Per}_T(X)} \lim_{k \to \infty} \mathbb{A}_k f(x, y) , \\
\gamma_{per}(f)	& = \sup_{y \in \operatorname{Per}_S(Y)} \inf_{x \in \operatorname{Per}_T(X)} \limsup_{k \to \infty} \mathbb{A}_k f(x, y) , \\
\delta_{per}(f)	& = \limsup_{k \to \infty} \sup_{y \in \operatorname{Per}_S(Y)} \inf_{x \in \operatorname{Per}_T(X)} \mathbb{A}_k f(x, y) .
\end{align*}
\end{Not}

It is elementary to observe that $\alpha_{per}(f) = \beta_{per}(f) = \gamma_{per}(f)$.

\begin{Lem}\label{Periodic measures are dense}
If $(X, T)$ has the periodic specification property, then $\mathcal{M}_T^{co}(X)$ is dense in $\mathcal{M}_T(X)$.
\end{Lem}

\begin{proof}
The result is originally due to \cite[Theorem 1]{SigmundGeneric}.
\end{proof}

\begin{Lem}\label{Some periodic identities}
	Consider topological dynamical systems $(X, T), (Y, S)$, and a continuous function $f \in C(X \times Y)$. Then:
	\begin{enumerate}[label = (\alph*)]
		\item If $(Y, S)$ has the periodic specification property, then $\alpha(f) \leq \alpha_{per}(f)$,
		\item If $(X, T)$ has the periodic specification property, then $\alpha_{per}(f) \leq \alpha(f)$,
		\item If $(X, T)$ has the periodic specification property, then $\beta_{per}(f) \leq \beta(f)$,
		\item If $(X, T)$ has the periodic specification property, then $\gamma_{per}(f) \leq \gamma(f)$,
		\item If the periodic points of $(X, T), (Y, S)$ are dense in their respective spaces, then $\delta(f) = \delta_{per}(f)$. In particular, this holds if both $(X, T), (Y, S)$ have the periodic specification property.
	\end{enumerate}
\end{Lem}

\begin{proof}
Let $D_X : (0, + \infty) \to \mathbb{N}$ be a function such that every $D_X(\eta)$-spaced specification $\xi = \left( x_i, a_i, b_i \right)_{i = 1}^n$ admits an $\eta$-tracing (with respect to $d_X$) of $\xi$ with period $b_n - a_1 + D_X(\eta)$. We also define $D_Y$ analogously.
	
	\begin{subproof}[Proof of claim (a)]
		Recall from Lemma \ref{Lower semi-continuity} that the map $\psi_f : \mathcal{M}_S(Y) \to \mathbb{R}$ is lower semi-continuous. Therefore, it follows from Lemma \ref{Periodic measures are dense} (applied to $(Y, S)$) that
		\begin{align*}
		\alpha(f)	=	\sup_{\nu \in \mathcal{M}_S(Y)} \psi_f(\nu)	= \sup_{\nu \in \mathcal{M}_S^{co}(Y)} \psi_f(\nu)	\leq \sup_{\nu \in \mathcal{M}_S^{co}(Y)} \inf_{\lambda \in \mathcal{M}_{T \times S}^{co}(X \times Y) \cap (\pi_Y)_*^{-1} \{\nu\}} \int f \mathrm{d} \lambda	& = \alpha_{per}(f) .
		\end{align*}
	\end{subproof}
	
	\begin{subproof}[Proof of claim (b)]
	Suppose that $(X, T)$ has the periodic specification property. Fix $\nu \in \mathcal{M}_S^{co}(Y)$. Our goal is to show that if $(X, T)$ has the periodic specification property, then
	\begin{align*}
	\psi_f(\nu)	& = \inf_{\lambda \in \mathcal{M}_{T \times S}^{co}(X \times Y) \cap (\pi_Y)_*^{-1} \{\nu\}} \int f \mathrm{d} \lambda .
	\end{align*}
	On one hand, the inequality that $\psi_f(\nu) \leq \inf_{\lambda \in \mathcal{M}_{T \times S}^{co}(X \times Y) \cap (\pi_Y)_*^{-1} \{\nu\}} \int f \mathrm{d} \lambda$ is trivial, so we set about proving the opposite inequality.
	
	Choose $y_0 \in Y, s_0 \in \mathbb{N}$ such that $S^{s_0} y_0 = y_0, \nu = \mathbb{A}_{s_0} \delta_{y_0}$, and fix $\epsilon > 0$. Choose $\lambda_0 \in \mathcal{M}_{T \times S}(X \times Y) \cap (\pi_Y)_*^{-1} \{\nu\}$ such that
		\begin{align*}
		\int f \mathrm{d} \lambda_0	& = \psi_f(\nu).
		\end{align*}
		Note that $\lambda_0$ is supported on $X \times \left\{ y_0, S y_0, \ldots, S^{s_0 - 1} y_0 \right\}$. By the pointwise ergodic theorem, there exists and $x_0 \in X$ such that $\lim_{k \to \infty} \mathbb{A}_k f (x_0, y_0)$ exists {and} $\lim_{k \to \infty} \mathbb{A}_k f (x_0, y_0) \leq \int f \mathrm{d} \lambda_0$. Choose $L_1 \in \mathbb{N}$ such that
		\begin{align*}
		\mathbb{A}_\ell f(x_0, y_0)	& < \lim_{k \to \infty} \mathbb{A}_k f(x_0, y_0) + \epsilon	& \textrm{for all $\ell \geq L_1$.}
		\end{align*}
		Choose $\eta > 0$ such that if $\mathbf{v}, \mathbf{w} \in X \times Y$, and $d_{X \times Y}(\mathbf{v}, \mathbf{w}) < \eta$, then $|f(\mathbf{v}) - f(\mathbf{w})| < \epsilon$. Set $D = D_X(\eta)$. For $\ell \in \mathbb{N}$, let $\xi^{(\ell)}$ be the specification
		$$
		\xi^{(\ell)} = \left( x_0 , [0, \ell - 1] \right) ,
		$$
		and let $x^{(\ell)}$ be an $(\ell - 1 + D)$-periodic $\eta$-tracing of $x_0$. For shorthand, write
		$$
		 \lambda^{(\ell)} = \lim_{k \to \infty} \mathbb{A}_k \delta_{\left(x^{(\ell)}, y_0\right)}  = \mathbb{A}_{\operatorname{lcm}(s_0, \ell - 1 + D)} \delta_{\left(x^{(\ell)}, y_0\right)} . $$
		Suppose $\ell \in s_0 \mathbb{N} + 1 - D$, meaning that $s_0$ divides $\ell - 1 + D$, and therefore $\operatorname{lcm}(s_0, \ell - 1 + D) = \ell - 1 + D$. Then
		\begin{align*}
			\int f \mathrm{d} \lambda^{(\ell)} - \mathbb{A}_\ell f \left( x^{(\ell)}, y_0 \right) =	& \left( \mathbb{A}_{\ell - 1 + D} f \left( x^{(\ell)}, y_0 \right) \right) - \left( \mathbb{A}_\ell f \left( x^{(\ell)}, y_0 \right) \right) \\
			=	& \left( \frac{\ell}{\ell - 1 + D} - 1 \right) \left( \frac{1}{\ell} \sum_{j = 0}^{\ell - 1} \mathbb{A}_\ell f \left( x^{(\ell)}, y_0 \right) \right) \\
			& + \left( \frac{1}{\ell - 1 + D} \sum_{j = \ell}^{\ell - 2 + D} f \left( T^j x^{(\ell)} , S^j y_0 \right) \right) \\
			=	& \left( \frac{1 - D}{\ell - 1 + D}\mathbb{A}_\ell f \left( x^{(\ell)}, y_0 \right) \right) \\
				& + \left( \frac{D - 1}{\ell - 1 + D} \mathbb{A}_{[\ell, \ell - 2 + D]} f \left( x^{(\ell)}, y_0 \right) \right) \\
			\Rightarrow \left| \int f \mathrm{d} \lambda^{(\ell)} - \frac{1}{\ell} \sum_{j = 0}^{\ell - 1} f \left( T^j x^{(\ell)} , S^j y_0 \right) \right| \leq 	& \frac{2 (D - 1)}{\ell - 1 + D} \|f\| \\
			 \leq	& \frac{2 (D - 1)}{\ell} \|f\| .
		\end{align*}
		Choose $L_2 \in \mathbb{N}$ such that $\frac{2 (D - 1)}{L_2} \|f\| < \epsilon$, and let $\ell_0 \in \mathbb{N}$ such that $\ell_0 \geq \max \{L_1, L_2\}$ and $s_0 \vert (\ell_0 - 1 + D)$. Then
		\begin{align*}
			\int f \mathrm{d} \lambda^{(\ell_0)}	& \leq \mathbb{A}_{\ell_0} f \left( x^{(\ell)}, y_0 \right) + \frac{2 (D - 1)}{\ell_0} \|f\| \\
			[\ell_0 \geq L_2]	& < \mathbb{A}_\ell f \left( x^{(\ell)}, y_0 \right) + \epsilon \\
			& < \mathbb{A}_\ell f \left( x_0, y_0 \right) + 2 \epsilon < \left( \lim_{k \to \infty} \mathbb{A}_k f \left( x_0, y_0 \right) + \epsilon \right) + 2 \epsilon \leq \int f \mathrm{d} \lambda_0 + 3 \epsilon = \psi_f(\nu) + 3 \epsilon .
		\end{align*}
		Therefore
		\begin{align*}
		\inf_{\lambda \in \mathcal{M}_{T \times S}^{co}(X \times Y) \cap (\pi_Y)_*^{-1} \{\nu\}} \int f \mathrm{d} \lambda	& \leq \int f \mathrm{d} \lambda^{(\ell_0)} \\
			& < \psi_f(\nu) + 3 \epsilon \\
		\stackrel{\epsilon \searrow 0}{\Rightarrow} \inf_{\lambda \in \mathcal{M}_{T \times S}^{co}(X \times Y) \cap (\pi_Y)_*^{-1} \{\nu\}} \int f \mathrm{d} \lambda	& \leq \psi_f(\nu) .
		\end{align*}
		As noted earlier in this subproof, the opposite inequality is trivial, so this is sufficient to prove that if $(X, T)$ has the periodic specification property, then
		\begin{align*}
			\psi_f(\nu)	& = \inf_{\lambda \in \mathcal{M}_{T \times S}^{co}(X \times Y) \cap (\pi_Y)_*^{-1} \{\nu\}} \int f \mathrm{d} \lambda	& \textrm{for all $\nu \in \mathcal{M}_{S}^{co}(Y)$.}
		\end{align*}
		We can then take a supremum over $\nu \in \mathcal{M}_S^{co}(Y)$ to conclude that
		\begin{align*}
		\alpha_{per}(f)	& = \sup_{\nu \in \mathcal{M}_{S}^{co}(Y)} \inf_{\lambda \in \mathcal{M}_{T \times S}^{co}(X \times Y) \cap (\pi_Y)_*^{-1} \{\nu\}} \int f \mathrm{d} \lambda = \sup_{\nu \in \mathcal{M}_{S}^{co}(Y)} \psi_f(\nu) \leq \sup_{\nu \in \mathcal{M}_S(Y)} \psi_f(\nu) = \alpha(f) .
		\end{align*}
	\end{subproof}
	
	\begin{subproof}[Proof of claims (c), (d)]
	These proofs are both similar to our proof of claim (b), using the periodic specification property to approximate the limiting behaviors of an orbit with a periodic orbit. Since the computations are all reducible to the same ``idea" found in our proof of claim (b), we omit them.
	\end{subproof}
	
	\begin{subproof}[Proof of claim (e)]
	For each $k \in \mathbb{N}$, the function $F_k : Y \to \mathbb{R} , y \mapsto \min_{x \in X} \mathbb{A}_k f(x, y)$ is continuous. Because the periodic points of $(Y, S)$ are dense, it follows that $\max_{y \in Y} F_k(Y) = \sup_{y \in \operatorname{Per}_S(Y)} F_k(y) $. Likewise, because the periodic points of $(X, T)$ are dense, for any $y \in Y$, we have that $F_k(y) = \inf_{x \in \operatorname{Per}_T (X)} \mathbb{A}_k f(x, y)$. Taken together, this tells us that
		\begin{align*}
			\sup_{y \in Y} \inf_{x \in X} \mathbb{A}_k f(x, y)	& = \sup_{y \in \operatorname{Per}_S (Y)} \inf_{x \in \operatorname{Per}_T(X)} \mathbb{A}_k f(x, y) \\
			\Rightarrow \limsup_{k \to \infty} \sup_{y \in Y} \inf_{x \in X} \mathbb{A}_k f(x, y)	& = \limsup_{k \to \infty} \sup_{y \in \operatorname{Per}_S (Y)} \inf_{x \in \operatorname{Per}_T(X)} \mathbb{A}_k f(x, y) \\
			\iff \delta(f)	& = \delta_{per}(f) .
		\end{align*}
	\end{subproof}
\end{proof}

\begin{Prop}\label{Delta limit exists for spec property}
Suppose that $(X, T)$ has the (not necessarily periodic) specification property. Then the limit $\lim_{k \to \infty} \max_{y \in Y} \min_{x \in X} \mathbb{A}_k f \left( x , y \right)$ exists.
\end{Prop}

In order to prove Proposition \ref{Delta limit exists for spec property}, we use Lemma \ref{Asymptotic Fekete}, which expands upon Fekete's ``subadditivity lemma" (c.f. \cite{Fekete}).

\begin{Lem}\label{Asymptotic Fekete}
Consider a sequence $(r_k)_{k = 1}^\infty$ of real numbers such that for every $\epsilon > 0$ exists $K = K(\epsilon) \in \mathbb{N}$ with the following property: for every $k \in \mathbb{N}$, if $k_1, \ldots, k_n$ are natural numbers for which $k_1, \ldots, k_n \geq K$ and $k = k_1 + \cdots + k_n$, then
\begin{align*}
	r_{k}	& \leq r_{k_1} + \cdots + r_{k_n} + k \epsilon .
\end{align*}
Then $\lim_{k \to \infty} \frac{r_k}{k}$ exists.
\end{Lem}

While other generalized subadditivity conditions exist in the literature which admit a version of Fekete's lemma (e.g. the De Bruijin-Erd\H{o}s condition, c.f. \cite{DeBruijinErdos, NearlySubadditive}), we know of no specific reference in the literature for Lemma \ref{Asymptotic Fekete}, so we sketch a proof here, even though it only deviates from the standard proof of Fekete's lemma in small and predictable ways.

\begin{proof}[Proof of Lemma \ref{Asymptotic Fekete}]
This result follows from a slightly modified version of the standard textbook proof of Fekete's lemma (c.f. \cite[Proposition 6.2.3]{TopologicalDimension}, to name only one example), so we only sketch our argument. Our goal is to prove that $\liminf_{k \to \infty} r_k / k = \limsup_{k \to \infty} r_k / k$, so choose $K \in \mathbb{N}$ large such that $r_K / K \approx \liminf_{k \to \infty} r_k / k$, and such that $r_k \lessapprox r_{k_1} + \cdots + r_{k_n}$ for $k_1, \ldots, k_n \geq K, k = k_1 + \cdots + k_n$. Consider $k \gg K$, and write
$$
k = \underbrace{K + \cdots + K}_{\textrm{$n$ times}} + (k - n K) ,
$$
where $n = \lfloor k / K \rfloor - 1$. This decomposes $k$ into a sum of integers $K, k - n K \in [K, 2 K - 1]$. Then $r_k / k \lessapprox r_K / K \approx \liminf_{k \to \infty} r_k / k$, meaning that $\limsup_{k \to \infty} r_k / k \leq \liminf_{k \to \infty} r_k / k$. The opposite inequality is trivial.
\end{proof}

\begin{proof}[Proof of Proposition \ref{Delta limit exists for spec property}]
It will suffice to prove that the sequence
$$r_k = \max_{y \in Y} \min_{x \in X} \mathbb{S}_k f \left( x , y \right)$$
satisfies the generalized subadditivity condition described in Lemma \ref{Asymptotic Fekete}, where $\mathbb{S}$ is as defined in Notation \ref{Def of average notation}. Fix $\epsilon > 0$.

Let $D_X : (0, + \infty) \to \mathbb{N}$ be a function such that every $D_X(\eta)$-spaced specification $\xi = \left( x_i, [a_i, b_i] \right)_{i = 1}^n$ on $X$ admits an $\eta$-tracing (with respect to $d_X$) of $\xi$. Choose $\eta > 0$ such that if $\mathbf{v}, \mathbf{w} \in X \times Y$, and $d_{X \times Y}(\mathbf{v}, \mathbf{w}) < \eta$, then $|f(\mathbf{v}) - f(\mathbf{w})| < \frac{\epsilon}{3}$, set $D = D_X(\eta)$, and choose $K \in \mathbb{N}$ such that
$$\frac{(D - 1) \|f\|}{K} < \frac{\epsilon}{3} $$
and $K > D$.
	
	Now fix $k \geq K$, and let $k_1, \ldots, k_n \geq K$ be such that $k = k_1 + \cdots + k_n$. Choose $y_0 \in Y$ such that
	\begin{align*}
		\max_{y \in Y} \min_{x \in X} \mathbb{S}_k f \left( x , y \right)	& = \min_{x \in X} \mathbb{S}_k f \left( x , y_0 \right) .
	\end{align*}
	For $i \in \{1, \ldots, n\}$, set
	\begin{align*}
		a_i	& = k_1 + \cdots + k_{i - 1} ,	& b_i	& = a_i + k_i - D ,
	\end{align*}
	where $a_1 = 0$, and choose $x_i \in X$ such that
	\begin{align*}
	\mathbb{S}_{k_i} f \left( x_i, S^{a_i} y_0 \right)	& = \min_{x \in X} \mathbb{S}_{k_i} f \left(x, S^{a_i} y_0 \right) \leq \max_{y \in Y} \min_{x \in X} \mathbb{S}_{k_i} f(x, y) = r_{k_i} .
	\end{align*}
	Consider the $D$-spaced specification $\xi = \left( T^{-a_i} x_i , [a_i, b_i] \right)_{i = 1}^n$, and let $x_0$ be an $\eta$-tracing of $\xi$. Then
	\begin{align*}
		r_k	& = \min_{x \in X} \mathbb{S}_{k} f(x, y_0) \\
			& \leq \mathbb{S}_k f(x_0, y_0) = \sum_{i = 1}^n \mathbb{S}_{[a_i, a_i + k_i - 1]} f \left( x_0 ,y_0 \right) = \sum_{i = 1}^n \left[ \left( \mathbb{S}_{[a_i, b_i]} f \left( x_0 ,y_0 \right) \right) + \left( \mathbb{S}_{[b_i + 1, b_i + D - 1]} f \left( x_0 ,y_0 \right) \right) \right] .
		\end{align*}
		We now isolate one estimate that will come up later in the computation as well. If $J_1, \ldots, J_n$ are integer intervals of length $\leq D - 2$, and $z_1, \ldots, z_n \in X \times Y$ for $i = 1, \ldots, n$ then
		\begin{align*}
		\addtocounter{equation}{1}\tag{\theequation}\label{Test estimate} \left| \sum_{i = 1}^n \mathbb{S}_{J_i} f \left( z_i \right) \right|	& \leq \sum_{i = 1}^n k_i \frac{|J_i| \cdot \|f\|}{k_i}	 \leq \sum_{i = 1}^n k_i \frac{(D - 1) \|f\|}{K}	\leq \sum_{i = 1}^n k_i \frac{\epsilon}{3}	= k \frac{\epsilon}{3} .
		\end{align*}
		Thus
		\begin{align*}
		r_k \leq	& \sum_{i = 1}^n \left[ \left( \mathbb{S}_{[a_i, b_i]} f \left( x_0 ,y_0 \right) \right) + \left( \mathbb{S}_{[b_i + 1, b_i + D - 1]} f \left( x_0 ,y_0 \right) \right) \right] \\
		[\textrm{inequality (\ref{Test estimate})}] \leq	& \sum_{i = 1}^n \left[ \mathbb{S}_{[a_i, b_i]} f \left( x_0 ,y_0 \right)  \right] + k \frac{\epsilon}{3} \\
		\leq	& \left[ \sum_{i = 1}^n \left( \mathbb{S}_{[a_i, b_i]} f \left( T^{-a_i} x_i, y_0 \right) + \frac{\epsilon}{3} \right) \right] + k \frac{\epsilon}{3} \\
		=	& \left[ \sum_{i = 1}^n \left( \mathbb{S}_{k_i - D} f \left( x_i, S^{a_i} y_0 \right) + \frac{\epsilon}{3} \right) \right] + k \frac{\epsilon}{3} \\
		\leq	& \left[ \sum_{i = 1}^n \mathbb{S}_{k_i - D} f \left( x_i, S^{a_i} y_0 \right) \right] + k \frac{2 \epsilon}{3} \\
		=	& \left[ \sum_{i = 1}^n \left( \left( \mathbb{S}_{k_i - 1} f \left( x_i, S^{a_i} y_0 \right) \right) - \left( \mathbb{S}_{[k_i - D + 1, k_i - 1]} f \left( x_i, S^{a_i} y_0 \right) \right) \right) \right] + k \frac{2 \epsilon}{3} \\
		[\textrm{inequality (\ref{Test estimate})}] \leq	& \left[ \sum_{i = 1}^n \mathbb{S}_{k_i - 1} f \left( x_i, S^{a_i} y_0 \right) \right] + k \epsilon \\
		=	& \left[ \sum_{i = 1}^n \min_{x \in X} \mathbb{S}_{k_i - 1} f \left( x, S^{a_i} y_0 \right) \right] + k \epsilon \\
		\leq	& \left[ \sum_{i = 1}^n \max_{y \in Y} \min_{x \in X} \mathbb{S}_{k_i - 1} f \left( x, y \right) \right] + k \epsilon \\
		=	& r_{k_1} + \cdots + r_{k_n} + k \epsilon .
	\end{align*}
This computation shows that $(r_k)_{k = 1}^\infty$ satisfies the conditions of Lemma \ref{Asymptotic Fekete}, from which it follows that the limit $ \lim_{k \to \infty} r_k / k = \lim_{k \to \infty} \max_{y \in Y} \min_{x \in X} \mathbb{A}_k f(x, y) $ exists.
\end{proof}

We note that \cite[Proposition 7]{IteratedOptimization} also provides sufficient conditions for the limit considered in Proposition \ref{Delta limit exists for spec property} to exist. At present, we know of no examples where this limit does not exist.

Our next major theorem of this section is the following.

\begin{Thm}\label{Identity for alpha and delta under periodic specification}
Let $(X, T), (Y, S)$ be topological dynamical systems, and consider $f \in C(X \times Y)$.
\begin{enumerate}[label = (\alph*)]
	\item If $(X, T), (Y, S)$ both have the periodic specification property, then $\alpha(f) \leq \delta(f)$.
	\item If $(Y, S)$ has the periodic specification property, then $\alpha(f) \geq \delta(f)$.
\end{enumerate}
In particular, if $(X, T), (Y, S)$ both have the periodic specification property, then $\alpha(f) = \delta(f)$.
\end{Thm}

First, we prove the following lemma which will be useful to prove Theorem \ref{Identity for alpha and delta under periodic specification}.

\begin{Lem}\label{Decomposition for marginal preimage of closed orbit measures}
	Let $y_0 \in \operatorname{Per}_S^{s_0}(Y)$, and write $\nu = \mathbb{A}_{s_0} \delta_{y_0} \in \mathcal{M}_S^{co}(Y)$. Then
	\begin{align*}
		\mathcal{M}_{T \times S}(X \times Y) \cap (\pi_Y)_*^{-1} \{\nu\}	& = \left\{ \mathbb{A}_{s_0} \left( \mu \times \delta_{y_0} \right) : \mu \in \mathcal{M}_{T^{s_0}}(X) \right\} .
	\end{align*}
Consequently, we have for all $f \in C(X \times Y)$ that
\begin{align*}
\psi_f(\nu)	& \geq \min_{x \in X} \frac{1}{s_0} \sum_{j = 0}^{s_0 - 1} f \left( T^j x, S^j y_0 \right) .
\end{align*}
\end{Lem}

\begin{proof}
We prove the first claim by proving inclusions in each direction. The forward inclusion
$\mathcal{M}_{T \times S}(X \times Y) \cap (\pi_Y)_*^{-1} \{\nu\} \subseteq \left\{ \mathbb{A}_{s_0} \left( \mu \times \delta_{y_0} \right) : \mu \in \mathcal{M}_{T^{s_0}}(X) \right\} $ can be proven by considering the disintegration of $\lambda \in \mathcal{M}_{T \times S} (X \times Y) \cap (\pi_Y)_*^{-1} \{\nu\}$ relative to $\pi$ (c.f. \cite[Chapter 5]{VianaOliveira}), i.e. $\lambda = \frac{1}{s_0} \sum_{j = 0}^{s_0 - 1} \mu_{S^j y_0} \times \delta_{S^j y_0}$, and prove that $\mu_{S^{j} y_0} = T_*^j \mu_{y_0}$ for $j \in \mathbb{Z}$. To prove the opposite inclusion, i.e. that a measure of the form $\mathbb{A}_{s_0} \left( \mu \times \delta_{y_0} \right)$ for $\mu \in \mathcal{M}_{T^{s_0}} (X)$ is $(T \times S)$-invariant and has $\nu$ as a marginal, is a routine calculation. Finally, to verify the inequality, consider $f \in C(X \times Y)$. Then
\begin{align*}
\psi_f(\nu)	& = \min_{\mu \in \mathcal{M}_{S^{s_0}}(X)} \int \mathbb{A}_{s_0} f(x, y_0) \mathrm{d} \mu(x) \geq \min_{x \in X} \mathbb{A}_{s_0} f(x, y_0) .
\end{align*}
\end{proof}

\begin{proof}[Proof of Theorem \ref{Identity for alpha and delta under periodic specification}]
	Throughout this proof, we will denote by $D_X : (0, + \infty) \to \mathbb{N}$ a function such that every $D_X(\eta)$-spaced specification $\xi = \left( x_i, [a_i, b_i] \right)_{i = 1}^n$ admits an $\eta$-tracing (with respect to $d_X$) of $\xi$ with period $b_n - a_1 + D_X(\eta)$. We use $D_Y$ analogously.
	
	\begin{subproof}[Proof of claim (a)]
		We know from Lemma \ref{Some periodic identities} that if $(Y, S)$ has the periodic specification property, then $\alpha(f) \leq \alpha_{per}(f) = \beta_{per}(f)$, so we'll show that $\beta_{per}(f) \leq \delta(f)$. Fix $y_0 \in \operatorname{Per}_S^{s_0}(Y)$, and set $s_n = s_0 n !$ for $n \in \mathbb{N}$, noting that $\operatorname{Per}_T^{s_1}(X) \subseteq \operatorname{Per}_T^{s_2} (X) \subseteq \cdots$ and $\operatorname{Per}_T(X) = \bigcup_{n = 1}^\infty \operatorname{Per}_T^{s_n}(X)$. Thus
		\begin{align*}
		\inf_{x \in \operatorname{Per}_T(X) } \lim_{k \to \infty} \mathbb{A}_k f(x, y_0) = \lim_{N \to \infty} \min_{x \in \operatorname{Per}_T^{s_N}(X)} \lim_{k \to \infty} \mathbb{A}_k f(x, y_0)	& = \lim_{N \to \infty} \min_{x \in \operatorname{Per}_T^{s_N}(X)} \mathbb{A}_{s_N} f(x, y_0) .
		\end{align*}
		We want to show that
		\begin{align*}
		\min_{x \in \operatorname{Per}_T^{s_N}(X)} \mathbb{A}_{s_N} f(x, y_0)	& \approx \min_{x \in X} \mathbb{A}_{s_N} f(x, y_0)	& \textrm{for large $N \in \mathbb{N}$.}
		\end{align*}
		Fix $\epsilon > 0$, and choose $\eta > 0$ such that if $\mathbf{v}, \mathbf{w} \in X \times Y$, and $d_{X \times Y}(\mathbf{v}, \mathbf{w}) < \eta$, then $|f(\mathbf{v}) - f(\mathbf{w})| < \epsilon / 2$. Set $D = D_X(\eta)$. For each $N \in \mathbb{N}$, choose $x_N \in X$ such that
		\begin{align*}
		 \min_{x \in X} \mathbb{A}_{s_N} f(x, y_0)	& = \mathbb{A}_{s_N} f(x_N, y_0) .
		\end{align*}
		Let $\xi^{(N)}$ be the specification
		$$\xi^{(N)} = \left( x_N, [0, s_N - D] \right) ,$$
		and choose $x^{(N)}$ to be an $\eta$-tracing of $\xi^{(N)}$ with period $s_N - D + D = s_N$. Then
		\begin{align*}
			& \min_{x \in \operatorname{Per}_T^{s_N}(X)} \mathbb{A}_{s_N} f(x, y_0) \\
			\leq	& \mathbb{A}_{s_N} f \left( x^{(N)}, y_0 \right) \\
			=	& \frac{s_N - D + 1}{s_N} \left( \mathbb{A}_{s_N - D + 1} f \left( x^{(N)}, y_0 \right) \right) + \frac{D - 1}{s_N} \left( \mathbb{A}_{[s_N - D , s_N - 1]} f \left( x^{(N)}, y_0 \right) \right) \\
			\leq	& \frac{s_N - D + 1}{s_N} \left( \mathbb{A}_{s_N - D + 1} f \left( x^{(N)}, y_0 \right) \right) + \frac{(D - 1) \|f\|}{s_N} \\
			\leq	& \frac{s_N - D + 1}{s_N} \left( \mathbb{A}_{s_N - D + 1} f \left( x_N, y_0 \right) + \frac{\epsilon}{2} \right) + \frac{(D - 1) \|f\|}{s_N} \\
			\leq	& \frac{s_N - D + 1}{s_N} \left( \mathbb{A}_{s_N - D + 1} f \left( x_N, y_0 \right) \right) + \frac{\epsilon}{2} + \frac{(D - 1) \|f\|}{s_N} \\
			=	& \left[ \mathbb{A}_{s_N} f(x_N, y_0) \right] - \frac{D - 1}{s_N} \left( \mathbb{A}_{[s_N - D + 1, s_N - 1]} f \left( x_N , y_0 \right) \right) + \frac{\epsilon}{2} + \frac{(D - 1) \|f\|}{s_N} \\
			\leq	& \left[ \min_{x \in X} \mathbb{A}_{s_N} f(x, y_0) \right] + \frac{\epsilon}{2} + \frac{2 (D - 1) \|f\|}{s_N} .
		\end{align*}
		Thus for $N \geq \max \left\{ D , \frac{4 (D - 1) \|f\|}{s_0 \epsilon} \right\}$, we have that
		\begin{align*}
			\inf_{x \in \operatorname{Per}_T(X)} \lim_{k \to \infty} \mathbb{A}_k f(x, y_0)	& \leq \min_{x \in \operatorname{Per}_T^{s_N}(X)} \mathbb{A}_{s_N} f(x, y_0) \\
			& \leq \min_{x \in X} \mathbb{A}_{s_N} f(x, y_0) + \epsilon \\
			\Rightarrow \sup_{y \in \operatorname{Per}_S^{s_0}(Y)} \inf_{x \in \operatorname{Per}_T(X)} \lim_{k \to \infty} \mathbb{A}_k f(x, y)	& \leq \sup_{y \in \operatorname{Per}_S^{s_0}(Y)} \min_{x \in X} \mathbb{A}_{s_N} f(x, y) + \epsilon \\
			& \leq \max_{y \in Y} \min_{x \in X} \mathbb{A}_{s_N} f(x, y) + \epsilon \\
			\Rightarrow \sup_{y \in \operatorname{Per}_S^{s_0}(Y)} \inf_{x \in \operatorname{Per}_T(X)} \lim_{k \to \infty} \mathbb{A}_k f(x, y)	& \leq \limsup_{N \to \infty} \max_{y \in Y} \min_{x \in X} \mathbb{A}_{s_N} f(x, y) + \epsilon \\
			& \leq \limsup_{k \to \infty} \max_{y \in Y} \min_{x \in X} \mathbb{A}_k f(x, y) + \epsilon \\
			& = \delta(f) + \epsilon .
		\end{align*}
		Therefore
		\begin{align*}
		\alpha(f) \leq \beta_{per}(f) = \sup_{s_0 \in \mathbb{N}} \sup_{y \in \operatorname{Per}_S^{s_0}(Y)} \inf_{x \in \operatorname{Per}_T(X)} \lim_{k \to \infty} \mathbb{A}_k f(x, y) \leq \delta(f) + \epsilon .
		\end{align*}
		Taking $\epsilon \searrow 0$ yields the inequality $\alpha(f) \leq \delta(f)$.
	\end{subproof}
	
	\begin{subproof}[Proof of claim (b)]
		Fix $\epsilon > 0$, and for each $k \in \mathbb{N}$, choose $y_k \in Y$ such that
		\begin{align*}
			\max_{y \in Y} \min_{x \in X} \mathbb{A}_k f(x, y) = \min_{x \in X} \mathbb{A}_k f(x, y_k) .
		\end{align*}
		Choose $\eta > 0$ such that if $\mathbf{v}, \mathbf{w} \in X \times Y$, and $d_{X \times Y}(\mathbf{v}, \mathbf{w}) < \eta$, then $|f(\mathbf{v}) - f(\mathbf{w})| < \epsilon$. Set $D = D_Y(\eta)$.
		
		For $k \geq D$, consider the specification in $(Y, S)$ given by
		\begin{align*}
		\xi^{(k)}	& = \left( y_k , [0, k - D] \right) .
		\end{align*}
		For each $k \in \mathbb{N}$, let $y^{(k)}$ be a $k$-periodic $\eta$-tracing of $\xi^{(k)}$. Write $\nu^{(k)} = \mathbb{A}_k \delta_{y^{(k)}} \in \mathcal{M}_S(Y)$. By Lemma \ref{Decomposition for marginal preimage of closed orbit measures} have for $k \geq D$ that
		\begin{align*}
		\alpha(f)	& = \sup_{\nu \in \mathcal{M}_S(Y)} \psi_f(\nu) \geq \psi_f \left( \nu^{(k)} \right) \geq \min_{x \in X} \mathbb{A}_k f \left( x, y^{(k)} \right) .
		\end{align*}
		For $x \in X$, we can estimate
		\begin{align*}
		\mathbb{A}_k f \left( x, y^{(k)} \right) =	& \left[ \mathbb{A}_k f \left( x, y_k \right) \right] + \frac{k - D + 1}{k} \left[ \mathbb{A}_{k - 1 + D} \left( f \left( x, y^{(k)} \right) - f \left( x, y_k \right) \right) \right] \\
			& + \frac{D - 1}{k} \left[ \mathbb{A}_{[k - D + 1, k - 1]} \left( f \left( x, y^{(k)} \right) f \left( x, y_k \right) \right) \right] \\
			\geq	& \mathbb{A}_k f \left( x, y_k \right) - \epsilon - \frac{(D - 1) \|f\|}{k} .
		\end{align*}
		Thus for $k \geq \max \left\{ \frac{(D - 1) \|f\|}{\epsilon}, D \right\}$, we have that
		\begin{align*}
		\min_{x \in X} \mathbb{A}_k f \left( x, y^{(k)} \right)	& \geq \min_{x \in X} \mathbb{A}_k f \left( x, y_k \right) - 2 \epsilon \\
			& = \max_{y \in Y} \min_{x \in X} \mathbb{A}_k f \left( x, y \right) - 2 \epsilon \\
		\Rightarrow \alpha(f)	& \geq \delta(f) - 2 \epsilon \\
		\stackrel{\epsilon \searrow 0}{\Rightarrow} \alpha(f)	& \geq \delta(f) .
		\end{align*}
	\end{subproof}
\end{proof}

\begin{Rmk}
While we will continue to use these kinds of specification arguments throughout this section, we will usually abbreviate them from here on. Since they all tend to have very similar ``steps," using specification to approximate one orbit with another that has some desired properties, and showing that these orbits are ``close" insofar as their ergodic averages are similar, we will usually use $O(\cdot)$ notation to denote the error terms that arise instead of making the error estimates explicit.
\end{Rmk}

Our next result shows that under certain specification hypotheses on $(Y, S)$, we can relate $\gamma$ and $\delta$ to each other.

\begin{Thm}\label{Delta gamma inequality}
Consider topological dynamical systems $(X, T), (Y, S)$ and a continuous function $f \in C(X \times Y)$.
\begin{enumerate}[label=(\alph*)]
	\item If $(Y, S)$ has the specification property, then $\delta(f) \leq \gamma(f)$.
	\item If $(X, T)$ has the specification property, then $\delta(f) \geq \gamma(f)$.
\end{enumerate}
\end{Thm}

\begin{proof}
Let $D_X : (0, + \infty) \to \mathbb{N}$ be a function such that every $D_X(\eta)$-spaced specification $\xi = \left( x_i, [a_i, b_i] \right)_{i = 1}^n$ admits an $\eta$-tracing (with respect to $d_X$) of $\xi$. We define $D_Y$ analogously.

\begin{subproof}[Proof of claim (a)]
	Choose $\eta > 0$ such that if $\mathbf{v}, \mathbf{w} \in X \times Y$, and $d_{X \times Y}(\mathbf{v}, \mathbf{w}) < \eta$, then $|f(\mathbf{v}) - f(\mathbf{w})| < \epsilon / 2$. Fix $D = D_Y(\eta / 2)$.
	
	Choose $K \in \mathbb{N}$ such that
	\begin{align*}
		\max_{y \in Y} \min_{x \in X} \mathbb{A}_K f(x, y)	& > \delta(f) - \epsilon ,	& K \geq	& \frac{(D - 1) \|f\|}{\epsilon},
	\end{align*}
	and set $y_0 \in Y$ such that
	\begin{align*}
		\max_{y \in Y} \min_{x \in X} \mathbb{A}_K f(x, y)	& = \min_{x \in X} \mathbb{A}_K f(x, y_0) .
	\end{align*}
	Note that $K \in \mathbb{N}$ can be taken arbitrarily large.
	
	Consider the specifications $\xi^{(n)} = \left( y_0 , [a_i, b_i] \right)_{i = 1}^{n}$, where
	\begin{align*}
		y_i	& = S^{-a_i} y_0 ,	& a_i	& = (i - 1) (K + D - 1),	& b_i	& = a_i + K - 1 .
	\end{align*}
	For each $n \in \mathbb{N}$, take $y^{(n)} \in Y$ to be a $\frac{\eta}{2}$-tracing of $\xi^{(n)}$. Because $Y$ is compact, there exists a subsequence of $\left( y^{(n)} \right)_{n = 1}^\infty$ converging to a point $y_0' \in Y$. Then $y_0'$ is an $\eta$-tracing of $\xi^{(n)}$ for all $n \in \mathbb{N}$, meaning that if $x_0 \in X, i \in \mathbb{N}, J_i = [a_i, b_i]$, then
	\begin{align*}
		\mathbb{A}_{J_i} f \left( x_0 , y_0' \right) \geq \mathbb{A}_{J_i} f \left( x_0 , S^{- a_i} y_0 \right) - \epsilon = \mathbb{A}_K f \left( T^{a_i} x_0, y_0 \right) - \epsilon	& \geq \left[ \min_{x \in X} \mathbb{A}_K f (x, y_0) \right] - \epsilon \\
			& \geq \delta(f) - 2 \epsilon . \addtocounter{equation}{1}\tag{\theequation}\label{Delta estimate}
	\end{align*}
	
	Fix $x \in X$. Then for $k \geq K$, we have
	\begin{align*}
		\mathbb{A}_k f \left( x, y_0' \right)	= & \frac{1}{k} \left[ \sum_{i = 1}^{\lfloor k / (K + D - 1) \rfloor} \mathbb{S}_{J_i} f \left( x , y_0 \right) \right] + O(1 / K) + O(K / k) \\
		[\textrm{inequality (\ref{Delta estimate})}] \geq	& \frac{1}{k} \left[ \sum_{i = 1}^{\lfloor k / (K + D - 1) \rfloor} K \left( \delta(f) - 2 \epsilon \right) \right] + O(1 / K) + O(K / k) \\
		=	& \delta(f) - 2 \epsilon + O(1 / K) + O(K / k) .
	\end{align*}
	In other words, there exist constants $A, B$ depending on $(X, T), (Y, S), f, \epsilon$ such that
	\begin{align*}
		\mathbb{A}_k f \left( x, y_0' \right)	& \geq \delta(f) - 2 \epsilon + A \frac{1}{K} + B \frac{K}{k} 
	\end{align*}
	for sufficiently large $k \in \mathbb{N}$. Therefore, taking $k \gg K \gg 1$, we can ensure that $\mathbb{A}_k f \left( x, y_0' \right) \geq \delta(f) - 3 \epsilon$ for sufficiently large $k \in \mathbb{N}$. Since this inequality holds for all sufficiently large $k \in \mathbb{N}$, it follows that
	\begin{align*}
		\limsup_{k \to \infty} \mathbb{A}_k f \left( x, y_0' \right)	& \geq \delta(f) - 3 \epsilon \\
		\Rightarrow \inf_{x \in X} \limsup_{k \to \infty} \mathbb{A}_k f \left( x, y_0' \right)	& \geq \delta(f) - 3 \epsilon .
	\end{align*}
	But
	\begin{align*}
		\gamma(f) = \sup_{y \in Y} \inf_{x \in X} \limsup_{k \to \infty} \mathbb{A}_k f(x, y) \geq \inf_{x \in X} \limsup_{k \to \infty} \mathbb{A}_k f \left( x, y_0' \right) \geq \delta(f) - 3 \epsilon ,
	\end{align*}
	and since this last inequality holds for all $\epsilon > 0$, we can take $\epsilon \searrow 0$ to conclude that $\gamma(f) \geq \delta(f)$.
\end{subproof}

\begin{subproof}[Proof of claim (b)]
Fix $\epsilon > 0 , y_0 \in Y$. By Proposition \ref{Delta limit exists for spec property}, there exists $K_1 \in \mathbb{N}$ such that if $k \geq K_1$, then $\max_{y \in Y} \min_{x \in X} \mathbb{A}_k f(x, y) < \delta(f) + \epsilon$.

Choose $\eta > 0$ such that if $\mathbf{v}, \mathbf{w} \in X \times Y$, and $d_{X \times Y}(\mathbf{v}, \mathbf{w}) < \eta$, then $|f(\mathbf{v}) - f(\mathbf{w})| < \epsilon$. Set $D = D_X(\eta / 2)$.

Consider $K \in \mathbb{N}, K \gg K_1$. For $i \in \mathbb{N}$, choose $x_i \in X$ such that
$$\min_{x \in X} \mathbb{A}_{[K(i - 1), K i - 1]} f \left( x, y_0 \right) = \mathbb{A}_{[K(i - 1), K i - 1]} f \left( x_i , y_0 \right) .$$
Consider the specification $\xi^{(n)}$ in $X$ given by
\begin{align*}
\xi^{(n)}	& = \left( x_i, J_i \right)_{i = 1}^n ,	& a_i	& = K(i - 1) D,	& b_i	& = K i - D ,	& J_i	& = [a_i, b_i] .
\end{align*}
Then for every $n \in \mathbb{N}$ exists an $\eta$-tracing $x^{(n)}$ of $\xi^{(n)}$. Take $x'$ to be any limit point of the sequence $\left( x^{(n)} \right)_{n = 1}^\infty$, meaning $x'$ is an $\eta$-tracing of $\xi^{(n)}$ for all $n \in \mathbb{N}$. For $k \in \mathbb{N}$, write
\begin{align*}
	n	& = \lfloor k / K \rfloor,	& r	& = k - K n \in \{0, 1, \ldots, K - 1\} .
\end{align*}
Then $k = K n + r$, and $n / k \leq 1 / K$. Consider $k \gg K$. Then
\begin{align*}
	& \mathbb{A}_k f \left( x', y_0 \right) \\
=	& \left[ \frac{1}{n} \sum_{i = 1}^{n} \mathbb{A}_{J_i} f \left( x' , y_0 \right) \right] + O \left( n / k \right) + O(K / k) \\
\leq	& \left[ \frac{1}{n} \sum_{i = 1}^{n} \mathbb{A}_{J_i} \left( f \left( x_i , y_0 \right) + \epsilon \right) \right] + O \left( 1 / K \right) + O(K / k) \\
=	& \left[ \frac{1}{n} \sum_{i = 1}^{n} \min_{x \in X} \mathbb{A}_{J_i} f \left( x , y_0 \right) \right] + \epsilon + O \left( 1 / K \right) + O(K / k) \\
\leq	& \left[ \frac{1}{n} \sum_{i = 1}^{n} \min_{x \in X} \mathbb{A}_{[K(i - 1), K i - 1]} f \left( x , y_0 \right) \right] + \epsilon + O \left( 1 / K \right) + O(K / k) \\
\leq	& \left[ \frac{1}{n} \sum_{i = 1}^{n} \max_{y \in Y} \min_{x \in X} \mathbb{A}_{K} f \left( x , y_0 \right) \right] + \epsilon + O \left( 1 / K \right) + O(K / k) \\
\leq	& \delta(f) + 2 \epsilon + O \left( 1 / K \right) + O(K / k) .
\end{align*}
In other words, there exist constants $A, B$ depending on $(X, T), (Y, S), f, \epsilon$ such that
\begin{align*}
\mathbb{A}_k f \left( x', y_0 \right)	& \leq \delta(f) + 2 \epsilon + A \frac{1}{K} + B \frac{K}{k} 
\end{align*}
for sufficiently large $k \in \mathbb{N}$. Therefore, taking $k \gg K \gg K_1$, we can ensure that $\mathbb{A}_k f \left( x', y_0 \right) \leq \delta(f) + 3 \epsilon$ for all sufficiently large $k \in \mathbb{N}$. Since this holds for all $\epsilon > 0$, we can conclude that
\begin{align*}
	\inf_{x \in X} \limsup_{k \to \infty} \mathbb{A}_k f \left( x, y_0 \right)	& \leq \limsup_{k \to \infty} \mathbb{A}_k f \left( x', y_0 \right) \\
	& \leq \delta(f) + 3 \epsilon \\
	\stackrel{\epsilon \searrow 0}{\Rightarrow} \inf_{x \in X} \limsup_{k \to \infty} \mathbb{A}_k f \left( x, y_0 \right)	& \leq \delta(f) \\
	\Rightarrow \gamma(f)	& \leq \delta(f) .
\end{align*}
\end{subproof}
\end{proof}

Of the two arguments we use in our proof of Theorem \ref{Delta gamma inequality}, neither is sufficient to compare $\delta(f)$ to $\beta(f)$. In our proof of claim (a), we produced a point $y_0 ' \in Y$ for which we could compare $\limsup_{k \to \infty} \mathbb{A}_k f \left(x, y_0' \right)$ with $\delta(f)$, but there is no guarantee that this point $y_0 '$ would admit an $x \in X$ such that $\lim_{k \to \infty} \mathbb{A}_k f \left( x, y_0' \right)$ exists. Similarly, in our proof of claim (b), we produce a point $x ' \in X$ for which we can compare $\limsup_{k \to \infty} \mathbb{A}_k f \left( x' , y_0 \right)$ with $\delta(f)$, but again, it does not follow from our proof that $\lim_{k \to \infty} \mathbb{A}_k f \left(x', y_0 \right)$ should exist.

\begin{Cor}
Consider topological dynamical systems $(X, T), (Y, S)$ and a continuous function $f \in C(X \times Y)$. If $(X, T), (Y, S)$ both have the periodic specification property, then $\alpha(f) = \gamma(f) = \delta(f) \leq \beta(f)$.
\end{Cor}

\begin{proof}
This corollary distills the conclusions of several earlier results. The identity $\alpha(f) = \delta(f)$ comes from Theorem \ref{Identity for alpha and delta under periodic specification}, the identity $\delta(f) = \gamma(f)$ comes from Theorem \ref{Delta gamma inequality}, and the inequality $\alpha(f) \leq \beta(f)$ comes from Theorem \ref{alpha leq beta}.
\end{proof}

\begin{Thm}\label{Bounding beta}
	Consider topological dynamical systems $(X, T), (Y, S)$ and a continuous function $f \in C(X \times Y)$. If $(Y, S)$ has the periodic specification property, and $\operatorname{Per}_T(X) \neq \emptyset$, then $\beta(f) \geq \delta(f)$. In particular, this attains if $(X, T), (Y, S)$ both have the periodic specification property.
\end{Thm}

\begin{proof}
	Let $D_Y : (0, + \infty) \to \mathbb{N}$ be a function such that every $D_Y(\eta)$-spaced specification $\xi = \left( y_i, [a_i, b_i] \right)_{i = 1}^n$ admits an $\eta$-tracing (with respect to $d_Y$) of $\xi$ with period $b_n - a_1 + D_Y(\eta)$. Fix $\epsilon > 0$, and choose $\eta > 0$ such that if $\mathbf{v}, \mathbf{w} \in X \times Y$, and $d_{X \times Y}(\mathbf{v}, \mathbf{w}) < \eta$, then $|f(\mathbf{v}) - f(\mathbf{w})| < \epsilon$. Set $D = D_Y(\eta)$.
	
	For every $\ell \in \mathbb{N}$, choose $y_\ell \in Y$ such that
	\begin{align*}
		\min_{x \in X} \mathbb{A}_\ell f(x, y_\ell)	& = \max_{y \in Y} \min_{x \in X} \mathbb{A}_\ell f(x, y) ,
	\end{align*}
	and let $\xi^{(\ell)}$ be the specification
	\begin{align*}
		\xi^{(\ell)}	& = \left( y_\ell , [0, \ell - 1] \right) .
	\end{align*}
	Then there exists an $\eta$-tracing $y^{(\ell)}$ of $\xi^{(\ell)}$ with period $\ell - 1 + D$.
	
	Our hypothesis that $\operatorname{Per}_T(X) \neq \emptyset$ means that if $y \in \operatorname{Per}_S(Y)$, then $\left\{ x \in X : (x, y) \in R_f \right\} \supseteq \operatorname{Per}_T(X)$, and in particular $\left\{ x \in X : (x, y) \in R_f \right\} \neq \emptyset$. Clearly, if $(X, T)$ has the periodic specification property, then $\operatorname{Per}_T(X) \neq \emptyset$. In particular, this implies that $\operatorname{Per}_S(Y) \subseteq \pi_Y R_f$.
	
	For now, set
	\begin{align*}
	s_\ell = \max_{y \in Y} \min_{x \in X} \mathbb{A}_\ell f(x, y) - \delta(f) ,
	\end{align*}
	noting that $\limsup_{\ell \to \infty} s_\ell = 0$, and therefore that $\liminf_{\ell \to \infty} |s_\ell| = 0$. For $\ell \in \mathbb{N}$, fix $x_0 \in X$ such that $\left( x_0 , y^{(\ell)} \right) \in R_f$. We will later introduce further assumptions on $\ell$ for reasons that will become clear from our estimates, but for now, observe that we can choose such an $\ell \in \mathbb{N}$ to be arbitrarily large. For $k \in \mathbb{N}$, set $n = \lfloor k / (\ell - 1 + D) \rfloor$, and $J_i = i (\ell - 1 + D) + [0, \ell - 1]$. Then we have that
	\begin{align*}
		\mathbb{A}_k f \left( x_0, y^{(\ell)} \right) =	& \left[ \frac{1}{k} \sum_{i = 0}^{n - 1} \sum_{j = 0}^{\ell - 2 + D} f \left( T^{j + i(\ell - 1 + D)} x_0 , S^{j + i(\ell - 1 + D)} y^{(\ell)} \right) \right] \\
			& + \left[ \frac{1}{k} \sum_{j = n (\ell - 1 + D)}^{k - 1} f \left( T^j x_0 , S^j y^{(\ell)} \right) \right] \\
		=	& \left[ \frac{1}{n} \sum_{i = 0}^{n - 1} \mathbb{A}_{J_i} f \left( x_0, y^{(\ell)} \right) \right] + O \left( \frac{1}{\ell} \right) + O \left( \frac{\ell}{k} \right) \\
		\geq	& \left[ \frac{1}{n} \sum_{i = 0}^{n - 1} \mathbb{A}_{J_i} f \left( x_0, S^{- i (\ell - 1 + D)} y_\ell \right) \right] - \epsilon + O \left( \frac{1}{\ell} \right) + O \left( \frac{\ell}{k} \right) \\
		=	& \left[ \frac{1}{n} \sum_{i = 0}^{n - 1} \mathbb{A}_{\ell} f \left( T^{i (\ell - 1 + D)} x_0, y_\ell \right) \right] - \epsilon + O \left( \frac{1}{\ell} \right) + O \left( \frac{\ell}{k} \right) \\
		\geq	& \left[ \frac{1}{n} \sum_{i = 0}^{n - 1} \min_{x \in X} \mathbb{A}_{\ell} f \left( x, y_\ell \right) \right] - \epsilon + O \left( \frac{1}{\ell} \right) + O \left( \frac{\ell}{k} \right) \\
		=	& \left[ \max_{y \in Y} \min_{x \in X} \mathbb{A}_{\ell} f \left( x, y\right) \right] - \epsilon + O \left( \frac{1}{\ell} \right) + O \left( \frac{\ell}{k} \right) \\
		\geq	& \delta(f) - |s_\ell| - \epsilon + O \left( \frac{1}{\ell} \right) + O \left( \frac{\ell}{k} \right) .
	\end{align*}
	That is, there exist constants $A, B \in \mathbb{R}$ depending on $(X, T), (Y, S), f, \epsilon$ such that
	\begin{align*}
	\mathbb{A}_k f \left( x_0, y^{(\ell)} \right)	& \geq \delta (f) - |s_\ell| - \epsilon + \frac{A}{\ell} + \frac{B \ell}{k}	& \textrm{for all $k, \ell \in \mathbb{N}$.}
	\end{align*}
	First, we can choose $L \in \mathbb{N}$ such that $|s_L| < \epsilon$, and such that $\left| A / L \right| \leq \epsilon$. We can then choose $K = K_L \in \mathbb{N}$ such that $\left| B L / K \right| \leq \epsilon$. This means that for $k \geq K$, we have for all $\left( x_0, y^{(L)} \right) \in R_f \cap (\pi_Y)^{-1} \left\{ y^{(L)} \right\}$ that
	\begin{align*}
	\mathbb{A}_k f \left( x_0, y^{(L)} \right)	& \geq \delta(f) - 4 \epsilon	& \textrm{for all $\left( x_0, y^{(L)} \right) \in R_f \cap (\pi_Y)^{-1} \left\{ y^{(L)} \right\} , k \geq K$} \\
	\Rightarrow \lim_{k \to \infty} \mathbb{A}_k f \left( x_0, y^{(L)} \right)	& \geq \delta(f) - 4 \epsilon .
	\end{align*}
	Since this inequality holds for all $\left( x_0, y^{(L)} \right) \in R_f \cap (\pi_Y)^{-1} \left\{ y^{(L)} \right\}$, it follows that
	\begin{align*}
	\beta(f)	& = \sup_{y \in \pi_Y R_f } \inf_{\left( x, y \right) \in R_f} \lim_{k \to \infty} \mathbb{A}_k f \left( x_0, y \right) \geq \inf_{\left( x, y^{(L)} \right) \in R_f} \lim_{k \to \infty} \mathbb{A}_k f \left( x_0, y^{(L)} \right) \geq \delta(f) - 4 \epsilon.
	\end{align*}
	Taking $\epsilon \searrow 0$ gives us the desired inequality.
\end{proof}

As Proposition \ref{Alpha delta counterexample} shows, some sort of condition on $(X, T), (Y, S)$ is needed to ensure that $\alpha, \beta, \gamma \geq \delta$, even for locally constant continuous functions, meaning in particular that Theorems \ref{Identity for alpha and delta under periodic specification}, \ref{Delta gamma inequality}, \ref{Bounding beta} require some kind of additional assumptions on the underlying systems. Both the systems $(X, T), (Y, S)$ in our proof of Proposition \ref{Alpha delta counterexample} are conjugate to shifts of finite type (defined in Definition \ref{Definition of SFT}), and by a classical result of P. Walters \cite[Theorem 1]{WaltersShadowing}, every shift of finite type has the shadowing property. Since this is the only time we mention the shadowing property, we decline to define it here, instead referring the interested reader to \cite[Section 7]{Panorama} for more on the shadowing property and its relation to various specification properties.

\section{Optimization of locally constant functions over shifts of finite type}\label{Locally constant secction}

In this section, we study the adversarial ergodic optimization of locally constant functions $f \in C(X \times Y)$ over transitive shifts of finite type $(X, T), (Y, S)$. In particular, we focus on the question of whether there exists an $\alpha$-maximizing measure $\nu \in \mathcal{M}_S (Y)$ (c.f. Definition \ref{Notation for psi}), and when $\nu$ can be taken to have finite support. We show that these questions are equivalent to qualitative questions about certain shifts $Y_{f, C} \subseteq Y$ whose points can be understood as generalized ground-state configurations. We reserve Subsection \ref{Locally constant functions, classical case, subsection} to focus on the ``classical ergodic optimization" case of this question (i.e. where $(X, T)$ is a trivial shift), a setting where the shifts $Y_{f, C}$ are relatively easy to understand, and where certain structural properties about these shifts can be proven. Our project of adversarial ergodic optimization is inspired in part by \cite{IteratedOptimization}, where the authors study the optimization of locally constant functions on shifts of finite type, so we consider this setting to be a natural starting point for the development of our theory.

\begin{Def}\label{Definition of SFT}
Consider a finite alphabet $\mathcal{A}$ with the discrete topology, and let $T : \mathcal{A}^\mathbb{Z} \to \mathcal{A}^\mathbb{Z}$ be the left shift action
$$(T x) (j) = x(j + 1) .$$
A \emph{shift on $\mathcal{A}$} is a nonempty compact $X \subseteq \mathcal{A}^\mathbb{Z}$ such that $T X = X$. A shift $(X, T)$ on $\mathcal{A}$ is called an \emph{$M$-step shift of finite type}, where $M \in \mathbb{N}$, if there exists a set $\mathcal{F} \subseteq \mathcal{A}^{\{0, 1, \ldots, M\}}$ such that
\begin{align*}
X	& = \left\{ x \in \mathcal{A}^\mathbb{Z} : \forall j \in \mathbb{Z} \; \left[ \left( T^j x \right) \vert_{[0, M]} \not \in \mathcal{F} \right] \right\} .
\end{align*}
More generally, we call a shift $X$ a \emph{shift of finite} type if it is an $M$-step shift of finite type for some $M \in \mathbb{N}$.
\end{Def}

Since the map $T$ for a shift $(X, T)$ is always understood to be the left shift action, we will sometimes omit reference to it, instead referring to $X$ as a shift over an alphabet $\mathcal{A}$.

\begin{Not}
Given a point $x \in \mathcal{A}^\mathbb{Z}$, and integers $a, b \in \mathbb{Z} \cup\{- \infty, + \infty\}; a \leq b$, we write
$$x \vert_a^b : = x \vert_{[a, b]} .$$
\end{Not}

\begin{Def}
	Given a shift $(X, T)$ on $\mathcal{A}$, a \emph{locally constant function} is a function $f : X \to \mathbb{R}$ such that there exists a constant $L \in \mathbb{N}$ for which if $x_1, x_2 \in X$, and $x_1 \vert_{-L}^L = x_2 \vert_{-L}^L$, then $f(x_1) = f(x_2)$. Equivalently, $f$ is locally constant if there exists a constant $L \in \mathbb{N}$ and a function $F : \mathcal{A}^{[-L, L]} \to \mathbb{R}$ such that
	$$f(x) = F \left( x \vert_{-L}^L \right) .$$
\end{Def}

We can observe that a locally constant function is continuous.

Throughout this section, we will make use of the following simplification. Consider $(X, T)$ an $M$-step shift of finite type over $\mathcal{A}$, and $f : X \to \mathbb{R}$ is a locally constant function. Consider a constant $L \in \mathbb{N}$ and a function $F : \mathcal{A}^{[-L, L] } \to \mathbb{R}$ such that $f (x) = F \left( x \vert_{-L}^L \right)$. Set $N = \max\{M, L\}$, and define a topological embedding map $c : \mathcal{A}^\mathbb{Z} \to \left( \mathcal{A}^{[-N, N] } \right)^\mathbb{Z}$ by
\begin{align*}
	(c x) (j)	& = \left( T^j x \right) \vert_{- N}^{N} .
\end{align*}
This is called a sliding block code (c.f. \cite{LindMarcus}). The image $X' = c(X) \subseteq \left( \mathcal{A}^{[-N, N] } \right)^\mathbb{Z}$ is a $1$-step shift of finite type over $\mathcal{A}^{[-N, N] }$, is topologically conjugate to $X$ via $c$, and admits a function $F' : \mathcal{A}^{[-N, N] } \to \mathbb{R}$ such that $\left( f \circ c^{-1} \right) \left( x' \right) = F' \left( x' (0) \right)$. Since it's always possible to do this, we will use this simplification for the remainder of this section, allowing us to omit some parameters and streamline certain definitions and arguments.

As such, for the remainder this section, we assume that $(X, T), (Y, S)$ are both $1$-step shifts of finite type over a finite alphabet $\mathcal{A}$, and that $f \in C(X \times Y)$ is such that there exists a function $F : \mathcal{A}^2 \to \mathbb{R}$ for which
$$f(x, y) = F(x(0), y(0)) .$$
{ A shift $(X, T)$ over an alphabet $\mathcal{A}$ is \emph{transitive} if there exists a constant $D \in \mathbb{N}$ such that if $x_1, x_2 \in X$, and $b_1, a_2 \in \mathbb{Z}, a_2 - b_1 \geq D$, then there exists $x' \in X$ such that
\begin{align*}
	x'(j)	& = x_1 (j)	& \textrm{for $j \leq b_1$,} \\
	x'(j)	& = x_2 (j)	& \textrm{for $j \geq a_2$.}
\end{align*}}
A shift of finite type $(X, T)$ is transitive if and only if it has the periodic specification property (c.f. \cite[Theorem 47 and associated references]{Panorama}). Therefore, it follows from Lemma \ref{Periodic measures are dense} that $\mathcal{M}_T^{co}(X)$ is dense in $\mathcal{M}_T(X)$ for transitive shifts of finite type $(X, T)$. In summary:

\begin{Lem}\label{Transitive SFT dense periodic measures}
A shift of finite type $(X, T)$ is transitive if and only if it has the periodic specification property. In particular, if $(X, T)$ is a transitive shift of finite type, then $\mathcal{M}_T^{co}(X)$ is dense in $\mathcal{M}_T(X)$.
\end{Lem}

\begin{Def}
	For $a, b \in \mathbb{Z}$, with $a \leq b$, write
	$$P_{a, b} : = \left\{ (v_1, v_2) \in \mathcal{A} : \exists x \in X \; \left[ \left( x(a) = v_1 \right) \land \left( x(b) = v_2 \right) \right] \right\} .$$
	For $v_1, v_2 \in P_{a, b}$ and $y \in Y$, write
	\begin{align*}
		H_{a, b, v_1, v_2} (y)	& : = \min \left\{ \mathbb{S}_{[a, b]} f \left( x , y \right) : x \in X, x (a) = v_1, x (b) = v_2 \right\} ,
	\end{align*}
	where $\mathbb{S}$ is as defined in Definition \ref{Def of average notation}. For a constant $C \geq 0$ and closed interval $[a, b] \subseteq \mathbb{Z}$ with $a \leq b$, we say that a point $y' \in Y$ is an \emph{$(f, C)$-improvement of $y$ on $[a, b]$} if
	\begin{align*}
		y \vert_{\mathbb{Z} \setminus [a, b]}	& = y' \vert_{\mathbb{Z} \setminus [a, b]} ,	&	& \textrm{and} \\
		H_{a, b, v_1, v_2}(y) + C	& < H_{a, b, v_1, v_2} \left( y' \right)	&	& \textrm{for all $(v_1, v_2) \in P_{a, b}$.}
	\end{align*}
	Write
	\begin{align*}
		Y_{f, C}	& = \bigcap_{a = - \infty}^{+ \infty} \bigcap_{b = a}^{+ \infty} \bigcup_{(v_1, v_2) \in P_{a, b}} \left\{ y \in Y : H_{a, b, v_1, v_2} (y) \geq \max_{\substack{y' \in Y \\ y' \vert_{\mathbb{Z} \setminus [a, b]} = y \vert_{\mathbb{Z} \setminus [a, b]}}} H_{a, b, v_1, v_2} \left( y' \right) - C \right\} .
	\end{align*}
	That is, a point $y \in Y$ is in $Y_{f, C}$ if and only if it cannot be $(f, C)$-improved on any $[a, b]$ with $a \leq b$.
\end{Def}

\begin{Prop}\label{Y_{f, C} closed}
$Y_{f, C}$ is a shift for all $C \geq 0$.
\end{Prop}

\begin{proof}
It's obvious that $Y_{f, C}$ is shift-invariant. To prove that $Y_{f, C}$ is closed, it suffices to observe that the functions
\begin{align*}
y	& \mapsto H_{a, b, v_1, v_2}(y) ,	& y	& \mapsto \max_{\substack{y' \in Y \\ y' \vert_{\mathbb{Z} \setminus [a, b]} = y \vert_{\mathbb{Z} \setminus [a, b]}}} H_{a, b, v_1, v_2} \left( y' \right)
\end{align*}
depend only on $y \vert_{a - 1}^{b + 1}$, which implies that the set
\begin{align*}
\left\{ y \in Y : H_{a, b, v_1, v_2} (y) \geq \max_{\substack{y' \in Y \\ y' \vert_{\mathbb{Z} \setminus [a, b]} = y \vert_{\mathbb{Z} \setminus [a, b]}}} H_{a, b, v_1, v_2} \left( y' \right) - C \right\}
\end{align*}
is closed for all $a, b \in \mathbb{Z} ; a \leq b ; (v_1, v_2) \in P_{a, b}$.

The final thing to check is that $Y_{f, C}$ is nonempty. This is a consequence of results which we will discuss later in this section, namely Theorem \ref{Theta map} and Proposition \ref{Optimizing measures supported on ground states}, so we prove it as Corollary \ref{Ground states nonempty}.
\end{proof}

\begin{Lem}\label{Continuity in C of Y_{f, C}}
Take $(C_n)_{n = 1}^\infty$ to be a decreasing sequence of nonnegative numbers converging to $C$. Then
	\begin{align*}
		Y_{f, C}	& = \bigcap_{n = 1}^\infty Y_{f, C_n} .
	\end{align*}
	In particular, if $Y_{f, C_n} \neq \emptyset$ for all $n \in \mathbb{N}$, then $Y_{f, C} \neq \emptyset$.
\end{Lem}

\begin{proof}
We prove the identity by a double-containment argument, i.e. by proving the claims $Y_{f, C} \subseteq \bigcap_{n = 1}^\infty Y_{f, C_n}$ and $\bigcap_{n = 1}^\infty Y_{f, C_n} \subseteq Y_{f, C}$.
	
First, consider $y_0 \in Y_{f, C}$. Then for every $a, b \in \mathbb{Z}; a \leq b$, there exists a pair $(v_1, v_2) \in P_{a, b}$ such that $H_{a, b, v_1, v_2} (y_0) \geq \max_{y' \in Y} H_{a, b, v_1, v_2} \left( y' \right) - C$. This means also that
$$H_{a, b, v_1, v_2} (y_0) \geq \max_{\substack{y' \in Y \\ y' \vert_{\mathbb{Z} \setminus [a, b]} = y \vert_{\mathbb{Z} \setminus [a, b]}}} H_{a, b, v_1, v_2} \left( y' \right) - C_n$$
for all $n \in \mathbb{N}$. Therefore
$$y_0 \in \bigcup_{(v_1, v_2) \in P_{a, b}} \left\{ y \in Y : H_{a, b, v_1, v_2} (y) \geq \max_{\substack{y' \in Y \\ y' \vert_{\mathbb{Z} \setminus [a, b]} = y \vert_{\mathbb{Z} \setminus [a, b]}}} H_{a, b, v_1, v_2} \left( y' \right) - C_n \right\}$$
for all $n \in \mathbb{N}, a \in \mathbb{Z}, b \in \mathbb{Z}, a \leq b$. That is,
\begin{align*}
y_0	& \in \bigcap_{n = 1}^\infty \bigcap_{a = - \infty}^{+ \infty} \bigcap_{b = a}^{+ \infty} \bigcup_{(v_1, v_2) \in P_{a, b}} \left\{ y \in Y : H_{a, b, v_1, v_2} (y) \geq \max_{\substack{y' \in Y \\ y' \vert_{\mathbb{Z} \setminus [a, b]} = y \vert_{\mathbb{Z} \setminus [a, b]}}} H_{a, b, v_1, v_2} \left( y' \right) - C_n \right\} = \bigcap_{n = 1}^\infty Y_{f, C_n} .
\end{align*}
Thus $Y_{f, C} \subseteq \bigcap_{n = 1}^\infty Y_{f, C_n}$.

In the other direction, suppose that $y_1 \in \bigcap_{n = 1}^\infty Y_{f, C_n}$, and fix $a, b \in \mathbb{Z}; a \leq b$. Then for every $n \in \mathbb{N}$ exists a pair $(v_1(n), v_2(n)) \in P_{a, b}$ such that
\begin{align*}
H_{a, b, v_1(n), v_2(n)} (y_1)	& \geq \max_{\substack{y' \in Y \\ y' \vert_{\mathbb{Z} \setminus [a, b]} = y \vert_{\mathbb{Z} \setminus [a, b]}}} H_{a, b, v_1(n), v_2(n)} \left( y' \right) - C_n .
\end{align*}
Since $n \mapsto (v_1(n), v_2(n)) $ is a map of an infinite set $\mathbb{N}$ into a finite set $P_{a, b} \subseteq \mathcal{A}^2$, it follows that there exists $(u_1, u_2) \in P_{a, b}$ and a sequence $n_1 < n_2 < n_3 < \cdots$ such that $(u_1, u_2) = (v_1(n_i), v_2(n_i))$ for all $i \in \mathbb{N}$. It follows then that
\begin{align*}
H_{a, b, u_1, u_2} (y_1)	& \geq \max_{\substack{y' \in Y \\ y' \vert_{\mathbb{Z} \setminus [a, b]} = y \vert_{\mathbb{Z} \setminus [a, b]}}} H_{a, b, u_1, u_2} \left( y' \right) - C_{n_i}	& \textrm{(for all $i \in \mathbb{N}$)} \\
\stackrel{i \to \infty}{\Rightarrow} H_{a, b, u_1, u_2} (y_1)	& \geq \max_{\substack{y' \in Y \\ y' \vert_{\mathbb{Z} \setminus [a, b]} = y \vert_{\mathbb{Z} \setminus [a, b]}}} H_{a, b, u_1, u_2} \left( y' \right) - C .
\end{align*}
Thus
$$y_1 \in \bigcup_{(v_1, v_2) \in P_{a, b}} \left\{ y \in Y : H_{a, b, v_1, v_2} (y) \geq \max_{\substack{y' \in Y \\ y' \vert_{\mathbb{Z} \setminus [a, b]} = y \vert_{\mathbb{Z} \setminus [a, b]}}} H_{a, b, v_1, v_2} \left( y' \right) - C \right\}$$
for all $a, b \in \mathbb{Z}; a \leq b$, and therefore
\begin{align*}
y_1	& \in \bigcap_{a = - \infty}^{+ \infty} \bigcap_{b = a}^{+ \infty} \bigcup_{(v_1, v_2) \in P_{a, b}} \left\{ y \in Y : H_{a, b, v_1, v_2} (y) \geq \max_{\substack{y' \in Y \\ y' \vert_{\mathbb{Z} \setminus [a, b]} = y \vert_{\mathbb{Z} \setminus [a, b]}}} H_{a, b, v_1, v_2} \left( y' \right) - C \right\} = Y_{f, C} .
\end{align*}
Thus $\bigcap_{n = 1}^\infty Y_{f, C_n} \subseteq Y_{f, C}$.

To prove the last claim, note that Proposition \ref{Y_{f, C} closed} implies that $\left( Y_{f, C_n} \right)_{n = 1}^\infty$ is a decreasing sequence of nonempty compact sets, and therefore has nonempty intersection.
\end{proof}

The following lemma will be useful in some later estimates, and provides an estimate on the difference between $H_{a, b, v_1, v_2} (y)$ and $\min_{x \in X} \mathbb{S}_{[a, b]} f(x, y)$.

\begin{Lem}\label{Minimum comparison estimate}
	Consider $1$-step shifts of finite type $(X, T), (Y, S)$ over a finite alphabet $\mathcal{A}$, and a locally constant function $f \in C(X \times Y)$ for which there exists $F : \mathcal{A}^2 \to \mathbb{R}$ such that
	\begin{align*}
		f(x, y)	& = F \left( x (0) , y (0) \right)	& \textrm{for all $(x, y) \in X \times Y$.}
	\end{align*}
	Suppose further that $(X, T)$ is transitive. Then there exists a constant $D_X \in \mathbb{N}$ depending on $(X, T)$, but not $f$ with the following property: If $a, b \in \mathbb{Z}; a \leq b; (v_1, v_2) \in P_{a, b}$, and $y_0 \in Y$, then
	\begin{align*}
		\min_{x \in X} \mathbb{S}_{[a, b]} f \left( x, y_0 \right) \leq H_{a, b, v_1, v_2} (y_0) \leq \left[ \min_{x \in X} \mathbb{S}_{[a, b]} f \left( x, y_0 \right) \right] + 4 D_X \|f\| .
	\end{align*}
\end{Lem}

\begin{proof}
	Take $D_X \in \mathbb{N}$ to be a constant such that if $x_1, x_2 \in X$, and $b_1, a_2 \in \mathbb{Z}, b_1 - a_2 \geq D_X$, then there exists $x' \in X$ such that
	\begin{align*}
		x'(j)	& = x_1 (j)	& \textrm{for $j \leq b_1$,} \\
		x'(j)	& = x_2 (j)	& \textrm{for $j \geq a_2$.}
	\end{align*}
	We consider two cases: where $b - a < 2 D_X$, and where $b - a \geq 2 D_X$.
	
	First, consider the case where $b - a < 2 D_X$, noting that $b - a + 1 \leq 2 D_X$. Then
	\begin{align*}
		\min_{x \in X} \mathbb{S}_{[a, b]} f(x, y_0) \leq H_{a, b, v_1, v_2} (y_0)	& \leq (b - a + 1) \|f\| \\
		& \leq 2 D_X \|f\| \\
		& \leq \left[ \min_{x \in X} \min_{x \in X} \mathbb{S}_{[a, b]} f \left( x, y_0 \right) \right] + (b - a + 1) \|f\| + 2 D_X \|f\| \\
		& \leq \left[ \min_{x \in X} \min_{x \in X} \mathbb{S}_{[a, b]} f \left( x, y_0 \right) \right] + 4 D_X \|f\| .
	\end{align*}
	
	Now consider the case where $b - a \geq 2 D_X$, noting that $b - D_X \geq a + D_X$. Choose $x_0, x_1 \in X$ as follows:
	\begin{align*}
		\mathbb{S}_{[a, b]} f \left( x_0, y_0 \right)	& = \min_{x \in X} \mathbb{S}_{[a, b]} f \left( x, y_0 \right) ,	& x_1 \vert_{a + D_X}^{b - D_X}	& = x_0 \vert_{a + D_X}^{b - D_X} ,	& x_1(a)	& = v_1,	& x_1(b)	& = v_2 .
	\end{align*}
	Then
	\begin{align*}
		\min_{x \in X} \mathbb{S}_{[a, b]} f \left( x, y_0 \right) \leq	& H_{a, b, v_1, v_2} (y_0) \\
		\leq	& \mathbb{S}_{[a, b]} f \left( x_1, y_0 \right) \\
		=	& \left[ \mathbb{S}_{[a, a + D_X - 1]} f(x_1, y_0) \right] + \left[ \mathbb{S}_{[a + D_X, b - D_X]} f(x_1, y_0) \right] + \left[ \mathbb{S}_{[b - D_X + 1, b]} f(x_1, y_0) \right] \\
		=	& \left[ \mathbb{S}_{[a, a + D_X - 1]} f(x_1, y_0) \right] + \left[ \mathbb{S}_{[a + D_X, b - D_X]} f(x_0, y_0) \right] + \left[ \mathbb{S}_{[b - D_X + 1, b]} f(x_1, y_0) \right] \\
		=	& \left[ \mathbb{S}_{[a, a + D_X - 1]} \left(f(x_1, y_0) - f(x_0, y_0) \right) \right] + \left[ \mathbb{S}_{[a, b]} f(x_0, y_0) \right] \\
		& + \left[ \mathbb{S}_{[b - D_X + 1, b]} \left(f(x_1, y_0) - f(x_0, y_0) \right) \right] \\
		\leq	& \left[ \mathbb{S}_{[a, b]} f \left( x_0, y_0 \right) \right] + 4 D_X \|f\| \\
		=	& \left[ \min_{x \in X} \mathbb{S}_{[a, b]} f \left( x, y_0 \right) \right] + 4 D_X \|f\| .
	\end{align*}
\end{proof}

Our goal in this section is to study the qualitative properties of $\alpha$-maximizing measures for locally constant functions $f \in C(X \times Y)$ when $(X, T), (Y, S)$ are transitive shifts of finite type. Theorem \ref{Theta map} provides a sufficient condition for these measures to exist. This is done by taking a (not necessarily locally constant) function $f \in C(X \times Y)$, and choosing a function $g \in C(Y)$ such that $\psi_f(\nu) = \int g \mathrm{d} \nu$ for all $\nu \in \mathcal{M}_S(Y)$, proving in particular that $\psi_f : \mathcal{M}_S(Y) \to \mathbb{R}$ is a continuous affine functional, and therefore has ergodic maximizers. This technique of reducing the adversarial ergodic optimization to a classical ergodic optimization can be generalized further, but that generalization will be the subject of a future work.

\begin{Thm}\label{Theta map}
	Take $(X, T), (Y, S)$ to be shifts, where $(X, T)$ is a transitive shift of finite type. For a continuous $f \in C(X \times Y)$, the functional $\psi_f : \mathcal{M}_S(Y) \to \mathbb{R}$ is continuous and affine. In particular, there exists $\nu_0 \in \partial_e \mathcal{M}_S(Y)$ which is $\alpha$-maximizing.
\end{Thm}

\begin{proof}
	We prove this result by constructing a continuous (but not necessarily linear) map $\Theta : C(X \times Y) \to C(Y) / B_S(Y)$ such that $\psi_f(\nu) = \int \Theta (f) \mathrm{d} \nu$ for all $f \in C(X \times Y), \nu \in \mathcal{M}_S(Y)$, where
	$$
	B_S(Y) = \overline{\operatorname{span} \left\{ g - g \circ S : g \in C(Y) \right\} } ,
	$$
	and the closure is taken in the uniform norm on $C(Y)$.
	
	For a function $f : X \times Y \to \mathbb{R}$ and integers $a \leq b$, we define the \emph{$[a, b]$-variation of $f$} to be the quantity
	\begin{align*}
		\mathbb{V}_\ell (f)	& : = \sup \left\{ \left| f(x_1, y_1) - f(x_2, y_2) \right| : (x_1, y_1), (x_2, y_2) \in X \times Y , (x_1, y_1) \vert_{-\ell}^\ell = (x_2, y_2) \vert_a^b \right\} .
	\end{align*}
	We can observe that if $\ell' \geq \ell$, then $\mathbb{V}_{\ell'} (f) \leq \mathbb{V}_\ell (f)$. Consequently, the sequence \linebreak$\left( \frac{1}{n + 1} \sum_{i = 0}^{n} \mathbb{V}_{i} (f) \right)_{n = 0}^\infty$ is non-increasing, and converges monotonically to $\lim_{n \to \infty} \mathbb{V}_{n} (f)$. Note that $f$ is continuous if and only if $\lim_{n \to \infty} \mathbb{V}_{n} ( f ) = 0$.
	
	{Take $D \in \mathbb{N}$ to be a constant such that if $x_1, x_2 \in X$, and $b_1, a_2 \in \mathbb{Z}, a_2 - b_1 \geq D$, then there exists $x' \in X$ such that
	\begin{align*}
		x'(j)	& = x_1 (j)	& \textrm{for $j \leq b_1$,} \\
		x'(j)	& = x_2 (j)	& \textrm{for $j \geq a_2$.}
	\end{align*}}
	Fix such a $D$ for the remainder of this proof, and let $f \in C (X \times Y)$ be a continuous function. For integers $a, b \in \mathbb{Z}; a \leq b$, we write
	$$
	\mathcal{L}_a^b (X) : = \left\{ x \vert_a^b \in \mathcal{A}^{ \{a, a + 1, \ldots, b\} } : x \in X \right\} .
	$$
	
	For $L \in \mathbb{N}$, we construct a map $w_f^{(L)} : \mathcal{L}_{-(L + D)}^{L + D}(Y) \to \mathcal{L}_{-(L + D)}^{L + D}(X)$ as follows: Consider a block $B \in \mathcal{L}_{- (L + D)}^{L + D} (Y)$, and choose $s_B \in Y$ such that $s_B \vert_{- (L + D)}^{L + D} = B$. Choose $t_B \in X$ such that
	\begin{align*}
		\mathbb{A}_{[-R_L, R_L]} f \left( t_B, s_B \right)	& = \min_{x \in X} \mathbb{A}_{[-R_L, R_L]} f \left( x, s_B \right) ,
	\end{align*}
	where $R_L = L + D$. Fix a letter $a_0 \in \mathcal{A}$ such that $\exists x \in X \; (x(0) = a_0)$. Let $\tilde{t}_B \in X$ be a point in $X$ such that
	\begin{align*}
		\tilde{t}_B \vert_{- L}^L	& = t_B \vert_{- L}^L ,	& \tilde{t}_B(- R_L)	& = a_0 ,	& \tilde{t}_B (R_L + 1)	& = a_0 .
	\end{align*}
	Write
	$$
	w_f^{(L)} (B) = \tilde{t}_B \vert_{ - R_L }^{R_L} \in \mathcal{L}_{-R_L}^{R_L} (X) .
	$$
	
	We now construct a map $ \theta_f^{(L)} : Y \to X \times Y$ as follows. For $n \in \mathbb{Z}, B \in \mathcal{L}_a^b (X)$, let $n \oplus B \in \mathcal{L}_{a - n}^{b - n} (X)$ be the block $\left( n \oplus B \right) (j) = B(j - n)$. Let $\tau_f^{(L)} : Y \to X$ be the map
	\begin{align*}
		\left( \tau_f^{(L)} (y) \right) (j)	& = w_f^{(L)} \left( Q_L n \oplus \left( y \vert_{Q_L n - R_L}^{Q_L n + R_L} \right) \right) (j - Q_L n) ,	& \textrm{for $n \in \Z, j \in [- R_L, R_L] + Q_L n$,}
	\end{align*}
	{where $Q_L = 2 R_L + 1$.} This means that for every $y \in Y$, there exists $y' \in Y$ such that $y' \vert_{R_L}^{R_L} = y \vert_{- R_L}^{R_L}$, along with $x' \in X$ such that $x' \vert_{- R_L + D}^{R_L - D} = \tau (y) \vert_{- R_L + D}^{R_L - D}$, for which
	\begin{align*}
		\mathbb{A}_{[-R_L, R_L]} f \left( x' , y' \right)	& = \min_{x \in X} \mathbb{A}_{[-R_L, R_L]} f \left( x , y' \right) .
	\end{align*}
	
	Fix $\nu \in \mathcal{M}_S(Y)$, and take $\lambda \in \mathcal{M}_{T \times S}(X \times Y) \cap (\pi_Y)_*^{-1} \{\nu\}$. To improve the readability of this computation, we set
	\begin{align*}
		\theta_f^{(L)}	& = \tau_f^{(L)} \times \operatorname{id}_Y , \\
		\Lambda_f^{(L)}(\nu)	& = \mathbb{A}_{[R_L, R_L]} \left( \theta_f^{(L)} \right)_* \nu	& \in \mathcal{M}_{T \times S}(X \times Y) \cap (\pi_Y)_*^{-1} \{\nu\} .
	\end{align*}
	Set $B (y) = y \vert_{- R_L}^{R_L}$, and let $s_B \in Y, t_B \in X, \tilde{t}_B \in X$ as earlier. Then
	\begin{align*}
		\int f \mathrm{d} \Lambda_f^{(L)}(\nu) - \int f \mathrm{d} \lambda =	& \int f \mathrm{d} \Lambda_f^{(L)}(\nu) - \frac{1}{Q_L} \sum_{j = - R_L}^{R_L} \int f \circ (T \times S)^j \mathrm{d} \lambda \\
		=	& \int \frac{1}{Q_L} \sum_{j = - R_L}^{R_L} \left[ f \left( (T \times S)^j \theta_f^{(L)} (y) \right) - f \left( T^j x, S^j y \right) \right] \mathrm{d} \lambda(x, y) \\
		\addtocounter{equation}{1}\tag{\theequation}\label{Term 1}	=	& \int \frac{1}{Q_L} \sum_{j = - R_L}^{R_L} \left[ f \left( (T \times S)^j \theta_f^{(L)} (y) \right) - f \left( T^j t_{B(y)} , S^j s_{B(y)} \right) \right] \mathrm{d} \lambda(x, y) \\
		\addtocounter{equation}{1}\tag{\theequation}\label{Term 2}	& + \int \frac{1}{Q_L} \sum_{j = - R_L}^{R_L} \left[ f \left( T^j t_{B(y)} , S^j s_{B(y)} \right) - f \left( T^j x, S^j s_{B(y)} \right) \right] \mathrm{d} \lambda(x, y) \\
		\addtocounter{equation}{1}\tag{\theequation}\label{Term 3}	& + \int \frac{1}{Q_L} \sum_{j = - R_L}^{R_L} \left[ f \left( T^j x , S^j s_{B(y)} \right) - f \left( T^j x, S^j y \right) \right] \mathrm{d} \lambda(x, y) .
	\end{align*}
	We estimate these three terms separately. For term (\ref{Term 1}), we can see that because $\left( t_B, s_B \right) \vert_{- L}^L = \theta_f^{(L)}(y) \vert_{- L}^L$, it follows that for $|j| \leq L$, we have that $(T \times S)^j \theta_f^{(L)} (y) \vert_{- L + |j|}^{L - |j|} = \left( t_B, s_B \right) \vert_{- L + |j|}^{L - |j|}$. Consequently, it follows that
	\begin{align*}
		& \left| \int \frac{1}{Q_L} \sum_{j = - R_L}^{R_L} \left[ f \left( (T \times S)^j \theta_f^{(L)} (y) \right) - f \left( T^j t_{B(y)} , S^j s_{B(y)} \right) \right] \mathrm{d} \lambda (x, y) \right| \\
		\leq	& \frac{1}{Q_L} \sum_{j = - R_L}^{R_L} \int \left| f \left( (T \times S)^j \theta_f^{(L)} (y) \right) - f \left( T^j t_{B(y)} , S^j s_{B(y)} \right) \right| \mathrm{d} \lambda (x, y) \\
		=	& \left[ \frac{1}{Q_L} \sum_{j = - L - D}^{- L - 1} \int \left| f \left( (T \times S)^j \theta_f^{(L)} (y) \right) - f \left( T^j t_{B(y)} , S^j s_{B(y)} \right) \right| \mathrm{d} \lambda (x, y) \right] \\
		& + \left[ \frac{1}{Q_L} \sum_{j = - L}^L \int \left| f \left( (T \times S)^j \theta_f^{(L)} (y) \right) - f \left( T^j t_{B(y)} , S^j s_{B(y)} \right) \right| \mathrm{d} \lambda (x, y) \right] \\
		& + \left[ \frac{1}{Q_L} \sum_{j = L + 1}^{L + D} \int \left| f \left( (T \times S)^j \theta_f^{(L)} (y) \right) - f \left( T^j t_{B(y)} , S^j s_{B(y)} \right) \right| \mathrm{d} \lambda (x, y) \right] \\
		\leq	& \frac{2 D}{Q_L} \|f\| + \left[ \frac{1}{Q_L} \sum_{j = - L}^L \mathbb{V}_{L - |j|} (f) \right] + \frac{2 D}{Q_L} \|f\| \\
		\leq	& \frac{4 D}{Q_L} \|f\| + \frac{4}{L + 1} \sum_{i = 0}^{L} \mathbb{V}_i (f) .
	\end{align*}
	For term (\ref{Term 2}), we see that
	\begin{align*}
		& \int \frac{1}{Q_L} \sum_{j = - R_L}^{R_L} \left[ f \left( T^j t_{B(y)} , S^j s_{B(y)} \right) - f \left( T^j x, S^j s_{B(y)} \right) \right] \mathrm{d} \lambda (x, y) \\
		=	& \int \frac{1}{Q_L} \sum_{j = - R_L}^{R_L} f \left( T^j t_{B(y)} , S^j s_{B(y)} \right) \mathrm{d} \lambda (x, y) - \int \frac{1}{Q_L} \sum_{j = - R_L}^{R_L} f \left( T^j x, S^j s_{B(y)} \right) \mathrm{d} \lambda (x, y) \\
		=	& \int \min_{x' \in X} \frac{1}{Q_L} \sum_{j = - R_L}^{R_L} f \left( T^j x' , S^j s_{B(y)} \right) \mathrm{d} \lambda (x, y) - \int \frac{1}{Q_L} \sum_{j = - R_L}^{R_L} f \left( T^j x, S^j s_{B(y)} \right) \mathrm{d} \lambda (x, y)	& \leq 0 .
	\end{align*}
	Finally, for term (\ref{Term 3}), we see that
	\begin{align*}
		& \left| \int \frac{1}{Q_L} \sum_{j = - R_L}^{R_L} \left[ f \left( T^j x , S^j s_{B(y)} \right) - f \left( T^j x, S^j y \right) \right] \mathrm{d} \lambda (x, y) \right| \\
		\leq	& \int \frac{1}{Q_L} \sum_{j = - R_L}^{R_L} \left| f \left( T^j x , S^j s_{B(y)} \right) - f \left( T^j x, S^j y \right) \right| \mathrm{d} \lambda (x, y) \\
		\leq	& \frac{1}{Q_L} \sum_{j = - R_L}^{R_L} \mathbb{V}_{R_L - |j|}(f)	& \leq \frac{4}{R_L + 1} \sum_{i = 0}^{R_L} \mathbb{V}_i (f) .
	\end{align*}
	Therefore, we see that
	\begin{align*}
		\int f \mathrm{d} \Lambda_f^{(L)}(\nu) - \int f \mathrm{d} \lambda	& \leq \frac{4 D}{Q_L} \|f\| + \left[ \frac{4}{L + 1} \sum_{i = 0}^{L} \mathbb{V}_i (f) \right] + \left[ \frac{4}{R_L + 1} \sum_{i = 0}^{R_L} \mathbb{V}_i (f) \right] \\
		& \leq \frac{4 D}{Q_L} \|f\| + 8 \left[ \frac{1}{L + 1} \sum_{i = 0}^{L} \mathbb{V}_i (f) \right] .
	\end{align*}
	
	For $\nu \in \mathcal{M}_{S} (Y)$, take $\lambda_0 \in \mathcal{M}_{T \times S}(X \times Y) \cap(\pi_Y)_*^{- 1} \{\nu\}$ such that $\int f \mathrm{d} \lambda_0 = \psi_f (\nu) $. Then
	\begin{align*}
		\int f \mathrm{d} \lambda_0 \leq \int f \mathrm{d} \Lambda_f^{(L)} (\nu) \leq \int f \mathrm{d} \lambda_0 + \frac{4 D}{Q_L} \|f\| + 8 \left[ \frac{1}{L + 1} \sum_{i = 0}^{L} \mathbb{V}_i (f) \right] .
	\end{align*}
	Thus
	\begin{align*}
		\left| \int f \mathrm{d} \Lambda_f^{(L)}(\nu) - \psi_f(\nu) \right|	& \leq \frac{4 D}{Q_L} \|f\| + 8 \left[ \frac{1}{L + 1} \sum_{i = 0}^{L} \mathbb{V}_i (f) \right] = o(1) ,
	\end{align*}
	where this estimate is uniform in $\nu \in \mathcal{M}_S (Y)$.
	
	It follows therefore that the sequence $\left( f^{(L)} \right)_{L = 1}^\infty = \left( \left( \mathbb{A}_{[-R_L, R_L]} f \right) \circ \theta_f^{(L)} \right)_{L = 1}^\infty$ is Cauchy with respect to the seminorm $p : C(Y) \to [0, + \infty)$ on $C(Y)$ given by
	\begin{align*}
	p (g) = \sup_{\nu \in \mathbb{R} \mathcal{M}_S(Y) \setminus \{0\}} \int g \mathrm{d} \nu / \|\nu\| = \sup_{\nu \in \mathcal{M}_S(Y)} \left| \int g \mathrm{d} \nu \right| ,
	\end{align*}
	since for $L_1, L_2 \in \mathbb{N}$, we have
	\begin{align*}
	p \left( f^{(L_1)} - f^{(L_2)} \right) =	& \sup_{\nu \in \mathcal{M}_S(Y)} \left| \int \left( f^{(L_1)} - f^{(L_2)} \right) \right| \mathrm{d} \nu \\
	\leq	& \left[ \sup_{\nu \in \mathcal{M}_S(Y)} \left| \int \left( f^{(L_1)} - \psi_f(\nu) \right) \mathrm{d} \nu \right| \right] + \left[ \sup_{\nu \in \mathcal{M}_S(Y)} \left| \int \left( \psi_f(\nu) - f^{(L_2)} \right) \mathrm{d} \nu \right| \right] \\
	\leq	& \frac{4 D}{Q_{L_1}} \|f\| + 8 \left[ \frac{1}{L_1 + 1} \sum_{i = 0}^{L_1} \mathbb{V}_i (f) \right] + \frac{4 D}{Q_{L_2}} \|f\| + 8 \left[ \frac{1}{L_2 + 1} \sum_{i = 0}^{L_2} \mathbb{V}_i (f) \right] \\
	\stackrel{\min\{L_1, L_2\} \to \infty}{\to}	& 0 .
	\end{align*}
	This seminorm $p$ is exactly the quotient norm on $C(Y) / B_S(Y)$, so the sequence $\left( f^{(L)} \right)_{L = 1}^\infty$ is Cauchy in $C(Y) / B_S(Y)$, and therefore convergent. Thus we can define a map $\Theta : C(X \times Y) \to C(Y) / B_S (Y)$ by $\Theta : f \mapsto \lim_{L \to \infty} f^{(L)}$, where this limit is taken in $C(Y) / B_S(Y)$. This map has the property that $\int \Theta (f) \mathrm{d} \nu = \psi_f(\nu)$ for all $\nu \in \mathcal{M}_S(Y)$. Therefore, the functional $\psi_f : \nu \mapsto \int \Theta (f) \mathrm{d} \nu$ is continuous and affine for fixed $f \in C(X \times Y)$. Our final claim follows from Choquet's theorem, since a continuous affine functional on a compact convex set attains a maximum on an extreme point.
\end{proof}

\begin{Prop}\label{Measures supported on ground states are optimizing}
	Consider $1$-step shifts of finite type $(X, T), (Y, S)$ over a finite alphabet $\mathcal{A}$, and a locally constant function $f \in C(X \times Y)$ for which there exists $F : \mathcal{A}^2 \to \mathbb{R}$ such that
	\begin{align*}
		f(x, y)	& = F \left( x (0) , y (0) \right)	& \textrm{for all $(x, y) \in X \times Y$.}
	\end{align*}
	Suppose that $(X, T), (Y, S)$ are both transitive. Consider further $C \geq 0$. If $\nu_0 \in \mathcal{M}_{S}(Y)$ is supported on $Y_{f, C}$, then $\nu_0$ is $\alpha$-maximizing.
\end{Prop}

\begin{proof}
{Throughout this proof, take $D_X \in \mathbb{N}$ to be a constant such that if $x_1, x_2 \in X$, and $b_1, a_2 \in \mathbb{Z}, a_2 - b_1 \geq D_X$, then there exists $x' \in X$ such that
\begin{align*}
	x'(j)	& = x_1 (j)	& \textrm{for $j \leq b_1$,} \\
	x'(j)	& = x_2 (j)	& \textrm{for $j \geq a_2$.}
\end{align*}}
	Fix such a $D_X$ for the remainder of this proof, noting that this $D_X$ is the constant of the same name from Lemma \ref{Minimum comparison estimate}. Fix $C \geq 0$.
	
	We break the proof into two cases, namely the case where $\nu_0$ is ergodic, and the case where $\nu_0$ is not necessarily ergodic.
	
	\textbf{The ergodic case:} We'll prove the contrapositive, i.e. that if $\nu_0$ is ergodic and $\psi_f(\nu_0) < \alpha(f)$, then $\operatorname{supp}(\nu_0) \setminus Y_{f, C} \neq \emptyset$. Suppose first that
	$$\psi_f(\nu_0) < \alpha(f) .$$
	Choose $\nu_1 \in \mathcal{M}_S^{co} (Y)$ such that
	$$\psi_f(\nu_0) < \psi_f(\nu_1) .$$
	This is possible because $\psi_f : \mathcal{M}_S(Y) \to \mathbb{R}$ is lower semi-continuous (Lemma \ref{Lower semi-continuity}), and $\mathcal{M}_S^{co}(Y)$ is dense in $\mathcal{M}_S(Y)$ (Lemma \ref{Transitive SFT dense periodic measures}). Take $y_1 \in Y, s_1 \in \mathbb{N}$ such that $S^{s_1} y_1 = y_1$ and $\nu_1 = \mathbb{A}_{s_1} \delta_{y_1}$. For each $k \in \mathbb{N}$, choose $x_k \in X$ such that
	\begin{align*}
	\mathbb{S}_k f(x_k, y_1)	& = \min_{x \in X} \mathbb{S}_k f(x, y_1) .
	\end{align*}
	Take $k_1 < k_2 < \cdots$ such that $\lim_{n \to \infty} \mathbb{A}_{k_n} f \left( x_{k_n}, y_1 \right) = \liminf_{k \to \infty} \mathbb{A}_k f \left( x_k, y_1 \right)$, and such that $\lambda_1 = \lim_{n \to \infty} \mathbb{A}_{k_n} \delta_{\left( x_{k_n}, y_1 \right)}$ exists. We claim
	$$\int f \mathrm{d} \lambda_1 = \psi_f(\nu_1) .$$
	Clearly $\int f \mathrm{d} \lambda_1 \geq \psi_f(\nu_1)$, so we'll prove the opposite inequality. Fix $\lambda \in \left( \partial_e \mathcal{M}_{T \times S}(X \times Y) \right) \cap (\pi_Y)_*^{-1} \{\nu_1\}$. Because $y_1 \in \operatorname{Per}_S(Y)$, we can choose $\left( x^{(\lambda)}, y_1 \right) \in \operatorname{supp}(\lambda)$ such that \linebreak$\lambda = \lim_{k \to \infty} \mathbb{A}_k \delta_{\left( x^{(\lambda)} , y_1 \right)}$. Then
	\begin{align*}
	\int f \mathrm{d} \lambda	& = \liminf_{k \to \infty} \mathbb{A}_k f \left( x^{(\lambda)} , y_1 \right) \geq \liminf_{k \to \infty} \min_{x \in X} \mathbb{A}_k f \left( x , y_1 \right) = \lim_{n \to \infty} \mathbb{A}_{k_n} f \left( x_{k_n} , y_1 \right) = \int f \mathrm{d} \lambda_1 .
	\end{align*}
	Lemma \ref{Sup on ergodic measures} then tells us that
	\begin{align*}
	\psi_f(\nu_1) = \inf \left\{ \int f \mathrm{d} \lambda : \lambda \in \left( \partial_e \mathcal{M}_{T \times S}(X \times Y) \right) \cap (\pi_Y)_*^{-1} \{\nu\} \right\} \geq \int f \mathrm{d} \lambda_1 .
	\end{align*}
	Therefore $\int f \mathrm{d} \lambda_1 = \psi_f(\nu_1)$.
	
	Next, choose $\lambda_0 \in \partial_e \mathcal{M}_{T \times S}(X \times Y) \cap (\pi_Y)_*^{-1} \{\nu_0\}$ such that
	$$\int f \mathrm{d} \lambda_0 = \psi_f(\nu_0) < \psi_f(\nu_1) .$$
	We can choose $(x_0, y_0) \in \operatorname{supp}(\lambda_0)$ such that $\lambda_0 = \lim_{k \to \infty} \mathbb{A}_k \delta_{\left(x_0, y_0\right)}$. Then
	\begin{align*}
	\nu_i	& = \lim_{k \to \infty} \mathbb{A}_k \delta_{y_i} ,	& y_i	& \in \operatorname{supp}(\nu_i)	& \textrm{for both $i = 0, 1$.}
	\end{align*}
	
	Set $\epsilon_0 = \frac{1}{2} \left( \int f \mathrm{d} \lambda_1 - \int f \mathrm{d} \lambda_0 \right) = \frac{1}{2} \left( \psi_f(\nu_1) - \psi_f(\nu_0) \right) > 0$, and fix $N_1 \in \mathbb{N}$ such that if $n \geq N_1$, then
	\begin{align*}
		\left[ \mathbb{A}_{k_n} f \left( x_{k_n}, y_1 \right) \right] - \left[ \mathbb{A}_{k_n} f \left( x_0, y_0 \right) \right]	& > \epsilon_0 .
	\end{align*}
	Take $N_2 = \left\lceil \frac{C + 1031 + 4 D_X \|f\|}{\epsilon_0} \right\rceil $, and set $N = \max \{N_1, N_2\} $. For $n \geq N$, consider a pair $(v_1, v_2) \in P_{0, k_n - 1}$. 
	\begin{align*}
		H_{0, k_n - 1, v_1, v_2} (y_0)	& \leq \left[ \min_{x \in X} \mathbb{S}_{k_n} f \left( x, y_0 \right) \right] + 4 D_X \|f\|	& \textrm{[Lemma \ref{Minimum comparison estimate}]} \\
		& \leq \left[ \mathbb{S}_{k_n} f(x_0, y_0) \right] + 4 D_X \|f\| \\
		& = k_n \left[ \mathbb{A}_{k_n} f \left( x_0, y_0 \right) \right] + 4 D_X \|f\| \\
		& < k_n \left( \left[ \mathbb{A}_{k_n} f \left( x_{k_n}, y_1 \right) \right] - \epsilon_0 \right) + 4 D_X \|f\|	& \textrm{[$n \geq N_1$]} \\
		& = \left[ \mathbb{S}_{k_n} f \left( x_{k_n}, y_1 \right) \right] + 4 D_X \|f\| - k_n \epsilon_0 \\
		& = \left[ \min_{x \in X} \mathbb{S}_{k_n} f \left( x_{k_n} , y_1 \right) \right] + 4 D_X \|f\| - k_n \epsilon_0 \\
		& \leq H_{0, k_n - 1, v_1, v_2}(y_1) + 4 D_X \|f\| - k_n \epsilon_0	& \textrm{[Lemma \ref{Minimum comparison estimate}]} \\
		& \leq H_{0, k_n - 1, v_1, v_2} (y_1) - C - 1031	& \textrm{[$k_n \geq n \geq N_2$]} \\
		& < H_{0, k_n - 1, v_1, v_2} (y_1) - C .
	\end{align*}
	Since these estimates were independent of $(v_1, v_2) \in P_{0, k_n - 1}$, it follows that $y_0$ is $(f, C)$-improved on $[0, k_N - 1]$ by $y_1$, meaning that $y_0 \in \operatorname{supp}(\nu_0) \setminus Y_{f, C}$.
	
	\textbf{The general case:} Consider now $\nu_0 \in \mathcal{M}_S(Y)$ such that $\nu_0(Y_{f, C}) = 1$. Consider the ergodic decomposition $\nu_0 = \int_{ \partial_e \mathcal{M}_S(Y) } \nu \mathrm{d} \rho(\nu)$, where $\rho$ is a Borel probability measure on $\partial_e \mathcal{M}_S(Y)$. If $\nu_0(Y_{f, C}) = 1$, then $\nu(Y_{f, C}) = 1$ for $\rho$-almost all $\nu \in \partial_e \mathcal{M}_S(Y)$. Thus, as we have already shown, we have that $\alpha(f) = \psi_f(\nu)$ for $\rho$-almost all $\nu \in \partial_e \mathcal{M}_S(Y)$. Finally, it follows from Theorem \ref{Theta map} that $\psi_f$ is continuous and affine, meaning that $\psi_f(\nu_0) = \int_{\partial_e \mathcal{M}_S(Y)} \psi_f(\nu) \mathrm{d} \rho(\nu) = \alpha(f)$.
\end{proof}

\begin{Prop}\label{Optimizing measures supported on ground states}
	Consider transitive $1$-step shifts of finite type $(X, T), (Y, S)$ over a finite alphabet $\mathcal{A}$, and a locally constant function $f \in C(X \times Y)$ for which there exists $F : \mathcal{A}^2 \to \mathbb{R}$ such that
	\begin{align*}
		f(x, y)	& = F \left( x (0) , y (0) \right)	& \textrm{for all $(x, y) \in X \times Y$.}
	\end{align*}
	Consider an invariant measure $\nu_0 \in \mathcal{M}_S (Y)$. For any $C \geq 0$, if $\nu_0 \in \mathcal{M}_S(Y)$ is $\alpha$-maximizing, then $\nu_0$ is supported on $Y_{f, C}$.
\end{Prop}

\begin{proof}
	{Throughout this proof, take $D_X \in \mathbb{N}$ to be a constant such that if $x_1, x_2 \in X$, and $b_1, a_2 \in \mathbb{Z}, a_2 - b_1 \geq D_X$, then there exists $x' \in X$ such that
	\begin{align*}
		x'(j)	& = x_1 (j)	& \textrm{for $j \leq b_1$,} \\
		x'(j)	& = x_2 (j)	& \textrm{for $j \geq a_2$.}
	\end{align*}}
	Fix such a $D_X$ for the remainder of this proof. Take $D_Y$ to be an analogous constant for $(Y, S)$, and set $D = \max \{ D_X, D_Y \}$, noting that $D$ will witness the transitivities of $X, Y, X \times Y$.
	
	We'll prove the contrapositive of the result, i.e. that if $\operatorname{supp}(\nu_0) \setminus Y_{f, C} \neq \emptyset$, then $\psi_f (\nu_0) < \alpha(f)$.
	
	Suppose that $\operatorname{supp}(\nu_0) \setminus Y_{f, C} \neq \emptyset$. Set $G_{f, C} = \operatorname{supp}(\nu_0) \setminus Y_{f, C}$, noting that $\nu_0(G_{f, C}) > 0$ and that $S G_{f, C} = G_{f, C}, S Y_{f, C} = Y_{f, C}$.
	
	Consider $y_1 \in G_{f, C}$. Then there exist $a, b \in \mathbb{Z}, a \leq b$ for which there exists an $(f, C)$-improvement $y_1'$ on $y_1$ on $[a, b]$. Assume without loss of generality that $a \geq 1$, since otherwise we could consider $S^{-(a + 1)} y_1$ in place of $y_1$. Take $U \subseteq Y$ to be the cylinder
	\begin{align*}
		U	& = \left\{ y \in Y : y \vert_{a - 1}^{b + 1} = y_1 \vert_{a - 1}^{b + 1} \right\} ,
	\end{align*}
	and fix an $(f, C)$-improvement $y_1'$ of $y_1$ on $[a, b]$, and set
	\begin{align*}
	\eta	& = \min_{(v_1, v_2) \in P_{a, b}} \left( H_{a, b, v_1, v_2}\left( y_1' \right) - H_{a, b, v_1, v_2}\left( y_1' \right) \right) > C \geq 0 .
	\end{align*}
	
	Since $U$ is an open neighborhood of $y_1 \in \operatorname{supp}(\nu_0)$, it follows that $\nu_0(U) > 0$. Moreover, we can see that $U \subseteq G_{f, C}$, since any point in $U$ admits an $(f, C)$-improvement on $[a, b]$. By applying the pointwise ergodic theorem to the probability measure-preserving dynamical system $\left( G_{f, C} , \nu_0, S^{b - a + 3} \right)$, we can infer the limit $\lim_{k \to \infty} \frac{1}{k} \sum_{j = 0}^{k - 1} \chi_U \left( S^{(b - a + 3) j} y \right)$ exists for $\nu_0$-almost all $y \in G_{f, C}$, and
	\begin{align*}
		\int \lim_{k \to \infty} \frac{1}{k} \sum_{j = 0}^{k - 1} \chi_U \left( S^{(b - a + 3) j} y \right) \mathrm{d} \nu_0(y)	& = \nu_0(U) > 0.
	\end{align*}
	Set $Q = b - a + 3$. Fix some $y_0 \in G_{f, C}$ such that $\lim_{k \to \infty} \frac{1}{k} \sum_{j = 0}^{k - 1} \chi_U \left( S^{Q j} y_0 \right) \geq \nu_0(U)$. We denote this limit by
	$$
	\sigma = \lim_{k \to \infty} \frac{1}{k} \sum_{j = 0}^{k - 1} \chi_U \left( S^{Q j} y_0 \right) \geq \nu_0(U) > 0 .
	$$
	Set
	$$I = \left\{ i \in \mathbb{N} \cup \{0\} : S^{Q i} y_0 \in U \right\} ,$$
	noting that $\frac{\left|Q I \cap [0, k - 1] \right|}{k} = \frac{\left|I \cap \left[0, \frac{k - 1}{Q} \right] \right|}{k} \to \sigma / Q > 0$.
	
	Because $Q > b - a + 2$, we can ensure that the intervals $\left\{ [a - 1, b + 1] + Q i : i \in I \right\}$ are separated enough that we can perform manipulations on them independently of each other. Construct a point $\tilde{y}_0 \in Y$ by
	\begin{align*}
		\tilde{y}_0 (j)	& = \begin{cases}
			y_1' (j - Q i)	& \textrm{if $i \in I,j \in [a - 1, b + 1] + Q i  $,} \\
			y_0 (j)	& \textrm{otherwise.}
		\end{cases}
	\end{align*}
	In other words, $\tilde{y}_0$ is made by taking $y_0$ and replacing every instance of the block $y_1 \vert_a^b$ with the block $y_1' \vert_a^b$, yielding an $(f, C)$-improvement where the replacement was made. We can ensure that $\tilde{y}_0 \in Y$ because $Y$ is a $1$-step shift of finite type, and
	\begin{align*}
	y_0(Q i + a - 1)	& = y_1'(a - 1) ,	& y_0 (Q i + b + 1)	& = y_1'(b + 1)	& \textrm{for $i \in I$.}
	\end{align*}
	
	Before proceeding with the proof, we offer a brief, informal sketch of our argument. The spirit of our argument is that $\tilde{y}_0$ is an $(f, C)$-improvement of $y_0$ on intervals of the form $[a, b] + Q i$ for $i \in I$, and moreover that $I$ has positive density. Consequently, if we were to compare $\lim_{k \to \infty} \mathbb{A}_k \delta_{\tilde{y}_0}$ (if the limit exists) to $\lim_{k \to \infty} \mathbb{A}_k \delta_{y_0}$, we would see that these improvements are made with positive ``frequency" (corresponding to the density of $Q I$), and thus that these improvements yield a strict inequality
	$$\psi_f(\nu_0) < \psi_f \left( \lim_{k \to \infty} \mathbb{A}_k \delta_{ \tilde{y}_0} \right).$$
	Note that our argument is slightly more nuanced than what we have just sketched, since instead of comparing $\nu$ to $\lim_{k \to \infty} \mathbb{A}_k \delta_{\tilde{y}_0}$ -a limit which needn't exist- we compare $\nu_0$ to a sequence of finitely supported, $S$-invariant approximations $\nu^{(k)} \in \mathcal{M}_S^{co}(Y)$ of $\mathbb{A}_k \delta_{\tilde{y}_0}$.
	
	For $k \in \mathbb{N}$, take $y^{(k)} \in Y$ to be a point of period $k \in \mathbb{N}$ such that
	\begin{align*}
		y^{(k)} \vert_0^{k - D}	& = \tilde{y}_0 \vert_0^{k - D} ,	& S^k y^{(k)}	& = y^{(k)} ,
	\end{align*}
	and set $\nu^{(k)} = \mathbb{A}_k \delta_{y^{(k)}} \in \mathcal{M}_S^{co}(Y)$.
	
	For each $k \in \mathbb{N}$, take $\lambda^{(k)} \in \mathcal{M}_{T \times S} (X \times Y) \cap (\pi_Y)_*^{-1} \left\{ \nu^{(k)} \right\}$ such that
	\begin{align*}
		\int f \mathrm{d} \lambda^{(k)} 	& = \psi_f \left( \nu^{(k)} \right) .
	\end{align*}
	Our goal is to show that $\psi_f \left( \nu^{(k)} \right) > \psi_f(\nu_0)$ for sufficiently large $k \in \mathbb{N}$.
	
	We know by Lemma \ref{Decomposition for marginal preimage of closed orbit measures} that $\int f \mathrm{d} \lambda^{(k)} \geq \min_{x \in X} \mathbb{A}_k f \left( x , y^{(k)} \right)$. For each $k \in \mathbb{N}$, choose $x^{(k)} \in X$ such that
	$$\min_{x \in X} \mathbb{A}_k f \left( x , y^{(k)} \right) = \mathbb{A}_k f \left( x^{(k)} , y^{(k)} \right) .$$
	We can observe here that for any $x \in X$, we have that
	\begin{align*}
		\mathbb{A}_k f \left( x, \tilde{y}_0 \right) =	& \left[ \frac{1}{k} \sum_{j = 0}^{k - D} f \left( T^j x, S^j \tilde{y}_0 \right) \right] + \left[ \frac{1}{k} \sum_{j = k - D + 1}^{k - 1} f \left( T^j x , S^j \tilde{y}_0 \right) \right] \\
		=	& \left[ \frac{1}{k} \sum_{j = 0}^{k - D} f \left( T^j x, S^j y^{(k)} \right) \right] + \left[ \frac{1}{k} \sum_{j = k - D + 1}^{k - 1} f \left( T^j x , S^j \tilde{y}_0 \right) \right] \\
		=	& \left[ \mathbb{A}_k f \left( x, y^{(k)} \right) \right] - \left[ \frac{1}{k} \sum_{j = k - D}^{k - 1} f \left( T^j x , S^j y^{(k)} \right) \right] + \left[ \frac{1}{k} \sum_{j = k - D + 1}^{k - 1} f \left( T^j x , S^j \tilde{y}_0 \right) \right] .
	\end{align*}
	Therefore
	\begin{align*}
	\addtocounter{equation}{1}\tag{\theequation}\label{Comparing minima for periodic approximation}	\left| \mathbb{A}_k f \left( x, \tilde{y}_0 \right) - \mathbb{A}_k f \left( x, y^{(k)} \right) \right|	& \leq \frac{2 (D - 1)}{k} \|f\|	& \textrm{for all $x \in X$.}
	\end{align*}
	
	Our goal is to compare the quantities $\min_{x \in X} \mathbb{A}_k f \left( x, y^{(k)} \right), \min_{x \in X} \mathbb{A}_k f \left( x, y_0 \right)$. Choose $x_k \in X$ for each $k \in \mathbb{N}$ such that
	$$\min_{x \in X} \mathbb{A}_k f \left( x, y_0 \right) = \mathbb{A}_k f \left( x_k, y_0 \right) ,$$
	and consider any limit point $\lambda = \lim_{n \to \infty} \mathbb{A}_{k_n} \delta_{\left( x_{k_n}, y_0 \right)}$ and observe that $\lambda \in \mathcal{M}_{T \times S}(X \times Y) \cap (\pi_Y)_*^{-1} \{\nu_0\}$. In particular, this implies that
	\begin{align*}
		\psi_f(\nu_0)	& \leq \liminf_{k \to \infty} \min_{x \in X} \mathbb{A}_k f \left( x, y_0 \right). \addtocounter{equation}{1}\tag{\theequation}\label{Estimate for minimizer of nu_0}
	\end{align*}
	
	Fix $k \in \mathbb{N}$. For each pair $(v_1, v_2) \in P_{a, b}$, take $w_{v_1, v_2} \in X$ such that
	\begin{align*}
	H_{a, b, v_1, v_2}(y_1)	& = \mathbb{S}_{[a, b]} f \left( w_{v_1, v_2}, y_1 \right) ,	& w_{v_1, v_2} (a)	& = v_1 ,	& w_{v_1, v_2} (b)	& = v_2 .
	\end{align*}
	Consider $\tilde{x}^{(k)} \in X$ defined as follows:
	\begin{align*}
		\tilde{x}^{(k)}(j)	& = \begin{cases}
			w_{x^{(k)}(a + Q i), x^{(k)}(b + Q i)}(j - Q i)	& \textrm{if $i \in I, j \in [a, b] + Q i$,} \\
			x^{(k)}(j)	& \textrm{otherwise.}
		\end{cases}
	\end{align*}
	That $\tilde{x}^{(k)}$ is in $X$ follows from the fact that $X$ is $1$-step. We can observe further that
	\begin{align*}
		& \mathbb{A}_k \left[ f \left( x^{(k)} , \tilde{y}_0 \right) - f \left( \tilde{x}^{(k)} , y_0 \right) \right] \\
		=	& \frac{1}{k} \sum_{i \in I , Q i \in [0, k - 1]} \left[ \left( \mathbb{S}_{Q i + [a, b]} f \left( x^{(k)} , \tilde{y}_0 \right) \right) - H_{a + Q i, b + Q i, x^{(k)} (a + Q i), x^{(k)} (b + Q i)} (y_0) \right] \\
		\geq	& \frac{1}{k} \sum_{i \in I , Q i \in [0, k - 1]} \left[  H_{a + Q i, b + Q i, x^{(k)}(a + Q i), x^{(k)}(b + Q i)} \left( \tilde{y}_0 \right) - H_{a + Q i, b + Q i, x^{(k)} (a + Q i), x^{(k)} (b + Q i)} (y_0) \right] \\
		=	& \frac{1}{k} \sum_{i \in I , Q i \in [0, k - 1]} \left[  H_{a, b, x^{(k)}(a + Q i), x^{(k)}(b + Q i)} \left( y_1' \right) - H_{a, b, x^{(k)} (a + Q i), x^{(k)} (b + Q i)} (y_1) \right] \\
		\geq	& \frac{1}{k} \sum_{i \in I , Q i \in [0, k - 1]} \eta \\
		=	& \frac{|Q I \cap [0, k - 1]|}{k} \eta .
	\end{align*}
	Therefore
	\begin{align*}
	\addtocounter{equation}{1}\tag{\theequation}\label{Gain of C in average} \mathbb{A}_k f \left( \tilde{x}^{(k)} , y_0 \right)	& \leq - \frac{\left| Q I \cap [0, k - 1] \right|}{k} \eta + \mathbb{A}_k f \left( x^{(k)} , \tilde{y}_0 \right) .
	\end{align*}
	Consider now the following estimates for $k \geq Q$:
	\begin{align*}
		\min_{x \in X} \mathbb{A}_k f \left( x , y_0 \right)
		\leq	& \mathbb{A}_k f \left( \tilde{x}^{(k)} , y_0 \right) \\
		\textrm{[Estimate (\ref{Gain of C in average})]} \leq	& - \frac{|Q I \cap [0, k - 1]|}{k} \eta + \mathbb{A}_k f \left( x^{(k)} , \tilde{y}_0 \right) \\
		\textrm{[Estimate (\ref{Comparing minima for periodic approximation})]} \leq	& - \frac{|Q I \cap [0, k - 1]|}{k} \eta + \left[ \mathbb{A}_k f \left( x^{(k)} , y^{(k)} \right) \right] + \frac{2 (D - 1)}{k} \|f\| \\
		=	& - \frac{|Q I \cap [0, k - 1]|}{k} \eta + \left[ \min_{x \in X} \mathbb{A}_k f \left( x , y^{(k)} \right) \right] + \frac{2 (D - 1)}{k} \|f\| \\
		\textrm{[Lemma \ref{Decomposition for marginal preimage of closed orbit measures}]}\leq	& \frac{2 (D - 1)}{k} \|f\| - \frac{|Q I \cap [0, k - 1]|}{k} \eta + \psi_f \left( \nu^{(k)} \right) .
	\end{align*}
	For convenience, write $r_k = \frac{|Q I \cap [0, k - 1]|}{k} \eta - \frac{2 (D - 1)}{k} \|f\|$, noting that $r_k \stackrel{k \to \infty}{\to} \frac{\sigma \eta}{Q}$. Take $0 < \epsilon_0 < \frac{\sigma \eta}{2 Q}$, and choose $K \in \mathbb{N}$ such that
	\begin{align*}
		\liminf_{k \to \infty} \min_{x \in X} \mathbb{A}_k f \left( x , y_0 \right)	& \leq \min_{x \in X} \mathbb{A}_K f \left( x , y_0 \right) + \epsilon_0 , \\
		\frac{\sigma \eta}{Q} - \epsilon_0	& < r_K .
	\end{align*}
	Then
	\begin{align*}
		\psi_f(\nu_0) \leq	& \liminf_{k \to \infty} \min_{x \in X} \mathbb{A}_k f \left( x , y_0 \right)	& \textrm{[Estimate (\ref{Estimate for minimizer of nu_0})]} \\
		\leq	& \left( \min_{x \in X} \mathbb{A}_K f \left( x , y_0 \right) \right) + \epsilon_0 \\
		\leq	& \psi_f \left( \nu^{(K)} \right) - r_K + \epsilon_0 \\
		<	& \psi_f \left( \nu^{(K)} \right) - \left( \frac{\sigma \eta}{Q} - \epsilon_0 \right) + \epsilon_0 \\
		=	& \psi_f \left( \nu^{(K)} \right) - \left( \frac{\sigma \eta}{Q} - 2 \epsilon_0 \right) \\
		<	& \psi_f \left( \nu^{(K)} \right) \\
		\leq	& \alpha(f) .
	\end{align*}
Thus $\psi_f(\nu_0) < \alpha(f)$.
\end{proof}

\begin{Cor}\label{Ground states nonempty}
Consider transitive shifts of finite type $(X, T), (Y, S)$ over a finite alphabet $\mathcal{A}$, and a locally constant function $f \in C(X \times Y)$. Then $Y_{f, C} \neq \emptyset$ for all $C \geq 0$.
\end{Cor}

\begin{proof}
By Theorem \ref{Theta map}, there exists $\nu_0 \in \mathcal{M}_S(Y)$ which is $\alpha$-maximizing. It follows from Proposition \ref{Optimizing measures supported on ground states} that $\operatorname{supp}(\nu_0) \subseteq Y_{f, C}$ for all $C \geq 0$.
\end{proof}

\begin{Cor}\label{Subordination principle}[Subordination principle]
Consider transitive shifts of finite type $(X, T), (Y, S)$, and a locally constant function $f \in C(X \times Y)$. Consider further $\nu_0, \nu_1 \in \mathcal{M}_S(Y)$ such that $\nu_0$ is $\alpha$-maximizing, and $\operatorname{supp} (\nu_1) \subseteq \operatorname{supp} (\nu_0)$. Then $\nu_1$ is $\alpha$-maximizing.
\end{Cor}

\begin{proof}
Proposition \ref{Optimizing measures supported on ground states} implies that $\operatorname{supp}(\nu_1) \subseteq \operatorname{supp}(\nu_0) \subseteq Y_{f, 1031}$, and Proposition \ref{Measures supported on ground states are optimizing} then implies that $\psi_f(\nu_1) = \alpha(f)$.
\end{proof}

\begin{Prop}\label{Periodic optimization}
Consider transitive $1$-step shifts of finite type $(X, T), (Y, S)$ over a finite alphabet $\mathcal{A}$, and a locally constant function $f \in C(X \times Y)$ for which there exists $F : \mathcal{A}^2 \to \mathbb{R}$ such that
\begin{align*}
	f(x, y)	& = F \left( x (0) , y (0) \right)	& \textrm{for all $(x, y) \in X \times Y$.}
\end{align*}
For each $C \geq 0$, the following conditions are equivalent:
\begin{enumerate}
	\item There exists $\nu_0 \in \mathcal{M}_S^{co}(Y)$ which is $\alpha$-maximizing.
	\item $Y_{f, C}$ contains a periodic point.
\end{enumerate}
\end{Prop}

\begin{proof}
(1)$\Rightarrow$(2): If $\nu_0 \in \mathcal{M}_S^{co}(Y)$ is such that $\psi_f(\nu_0) = \alpha(f)$, then Proposition \ref{Optimizing measures supported on ground states} implies that $\operatorname{supp}(\nu_0) \subseteq Y_{f, C}$, and $\operatorname{supp}(\nu_0)$ is the orbit of a periodic point in $Y_{f, C}$.

(2)$\Rightarrow$(1): If $Y_{f, C}$ contains a periodic point $y_0$, then $\nu_0 = \lim_{k \to \infty} \mathbb{A}_k \delta_{y_0}$ is supported on the orbit of $y_0$, and thus $\operatorname{supp}(\nu_0) \subseteq Y_{f, C}$. It then follows from Proposition \ref{Measures supported on ground states are optimizing} that $\psi_f(\nu_0) = \alpha(f)$.
\end{proof}

\subsection{The case where $(X, T)$ is trivial}\label{Locally constant functions, classical case, subsection}

Consider now the case where $X = \{\star\}$ is a shift on $1$ element, which we denote by $\star$, so that $(X, T)$ is the $1$-point topological dynamical system. We can note that $X$ is a transitive $1$-step shift of finite type. In this context, we can write $H_{a, b, v_1, v_2}(y) = \mathbb{S}_{[a, b]} f \left( \star, y \right)$, and characterize $Y_{f, C}$ for $C \geq 0$ as
\begin{align*}
Y_{f, C}	& = \bigcap_{a = - \infty}^{+ \infty} \bigcap_{b = a}^{+ \infty} \left\{ y \in Y : \mathbb{S}_{[a, b]} f \left( \star , y \right) \geq \max_{\substack{y' \in Y \\ y' \vert_{\mathbb{Z} \setminus [a, b]} = y \vert_{\mathbb{Z} \setminus [a, b]}}} \mathbb{S}_{[a, b]} f \left( \star, y' \right) - C \right\} .
\end{align*}
In other words, the elements of $Y_{f, C}$ are the set of points $y \in Y$ such that if $y' \in Y$ and $\left| \left\{ j \in \mathbb{Z} : y(j) \neq y'(j) \right\} \right| < \infty$, then $\mathbb{S}_{\mathbb{Z}} \left( f \left( \star, y \right) - f \left( \star, y' \right) \right) \geq - C$. For $C = 0$, these points can be understood as ground-state configurations (c.f. \cite[Definition 2.3]{GroundStates}).

In addition, when $(X, T) = \left( \star, \operatorname{id}_{\{\star\}} \right)$, these $Y_{f, 0}$ shifts are topological Markov fields \cite[Definition 3.1]{MarkovFields}, and therefore sofic (c.f. \cite[Proposition 3.5]{MarkovFields}). In particular, this means that $Y_{f, 0}$ contains periodic points, and by extension, so does $Y_{f, C}$ for all $C \geq 0$.

Even in this setting where $(X, T)$ is trivial, the shift $Y_{f, C}$ is in general \emph{not} of finite type. The following example was communicated to us by A. Hru\u{s}kov\'{a} \cite{Aranka}. Consider the $1$-step shift of finite type $Y$ defined as follows. Let the finite digraph $\mathcal{G} = (\mathcal{V}, \mathcal{E})$ with edge weight $F : \mathcal{E} \to \mathbb{R}$ be as illustrated below:
$$
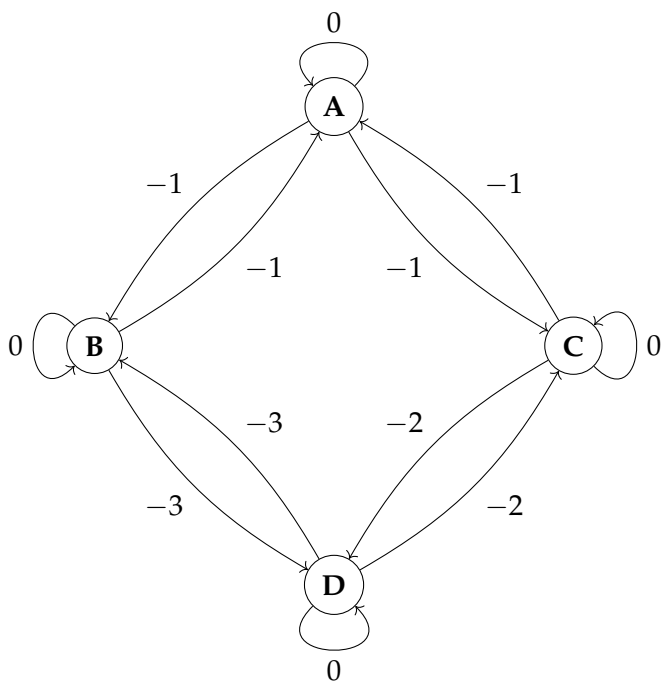
\begin{tikzpicture}[node distance={30mm}, main/.style = {draw, circle}]
	\node[main] (1) {$\mathbf{A}$};
	\node[main] (2) [below left of =1] {$\mathbf{B}$};
	\node[main] (3) [below right of =1] {$\mathbf{C}$};
	\node[main] (4) [below left of =3] {$\mathbf{D}$};
	\draw[->] (1) to [out=210, in=60] node[midway, above left] {$-1$} (2);
	\draw[->] (2) to [out=30, in=240] node[midway, below right] {$-1$} (1);
	\draw[->] (1) to [out=300, in=150] node[midway, below left] {$-1$} (3);
	\draw[->] (3) to [out=120, in=330] node[midway, above right] {$-1$} (1);
	\draw[->] (3) to [out=210, in=60] node[midway, above left] {$-2$} (4);
	\draw[->] (4) to [out=30, in=240] node[midway, below right]{$-2$} (3);
	\draw[->] (2) to [out=300, in=150] node[midway, below left] {$-3$} (4);
	\draw[->] (4) to [out=120, in=330] node[midway, above right] {$-3$} (2);
	\draw[->] (1) to [out=45, in=135, looseness=5] node[midway, above] {$0$} (1);
	\draw[->] (2) to [out=135, in=225, looseness=5] node[midway, left] {$0$} (2);
	\draw[->] (3) to [out=315, in=45, looseness=5] node[midway, right] {$0$} (3);
	\draw[->] (4) to [out=225, in=315, looseness=5] node[midway, below] {$0$} (4);
\end{tikzpicture}
$$
This digraph induces a $1$-step shift of finite type $Y$ on the alphabet $\mathcal{A} = \mathcal{E}$, where $Y \subseteq \mathcal{A}^\mathbb{Z}$ is the space of all bi-infinite walks $(e(j))_{j \in \mathbb{Z}} \in \mathcal{E}^\mathbb{Z}$ through $\mathcal{G}$. This constitutes a $1$-step shift of finite type $(Y, S)$ (c.f. \cite[Proposition 2.2.6]{LindMarcus}), and moreover, it can be verified that the shift is transitive. Consider the function $f : X \times Y \to \mathbb{R}$ given by $f(x, y) = F(y(0))$.

Fix $C \in [0, 1)$, and assume for contradiction that $Y_{f, C}$ is an $M$-step shift of finite type for some $M \in \mathbb{N}$. Consider the point $y_M \in Y$ defined as follows:
\begin{align*}
y_M(j)	& = \begin{cases}
\mathbf{AA}	& \textrm{if $j \leq -1$,} \\
\mathbf{AB}	& \textrm{if $j = 0$,} \\
\mathbf{BB}	& \textrm{if $1 \leq j \leq M$,} \\
\mathbf{BD}	& \textrm{if $j = M + 1$,} \\
\mathbf{DD}	& \textrm{if $j \geq M + 2$.}
\end{cases}
\end{align*}
We observe firstly that $y_M$ admits an $(f, C)$-improvement on $[0, M + 2]$. Consider the point $y_M' \in Y$ given by
\begin{align*}
y_M'(j)	& = \begin{cases}
\mathbf{AA}	& \textrm{if $j \leq -1$,} \\
\mathbf{AC}	& \textrm{if $j = 0$,} \\
\mathbf{CC}	& \textrm{if $1 \leq j \leq M$,} \\
\mathbf{CD}	& \textrm{if $j = M + 1$,} \\
\mathbf{DD}	& \textrm{if $j \geq M + 2$.}
\end{cases}
\end{align*}
Then $H_{0, M + 2, v_1, v_2} (y_M) = - 4, H_{0, M + 2, v_1, v_2} \left( y_M ' \right) = - 3$. Therefore $y_M \not \in Y_{f, C}$. However, the point $y_M$ contains no $Y_{f, C}$-illegal words of length $M + 1$, since for every $h \in \mathbb{Z}$ exists $y \in Y_{f, C}$ such that $S^h y \vert_0^M = S^h y_M \vert_0^M$. Therefore $Y_{f, C}$ cannot be an $M$-step shift of finite type. Since our choice of $M \in \mathbb{N}$ was arbitrary, this means that $Y_{f, C}$ is not of finite type.

We conclude this section by proposing the following conjecture.

\begin{Conj}\label{Periodic optimization conjecture}
Take $(X, T), (Y, S)$ to be transitive shifts of finite type, and let $f \in C(X \times Y)$ be a locally constant function on $X \times Y$. Then there exists $\nu \in \mathcal{M}_S^{co}(Y)$ which is $\alpha$-maximizing.
\end{Conj}

A related question to Conjecture \ref{Periodic optimization conjecture} is when we can ``algorithmically" describe a measure $\nu \in \mathcal{M}_S^{co}(Y)$ for which $\alpha(f) = \psi_f(\nu)$, i.e. when we can explicitly exhibit such a measure based on information about the shifts of finite type $X, Y$ and the locally constant function $f$. This is in turn related to a broader question of when $\alpha(f)$ can be computed explicitly, and will be part of a future work.

\bibliography{/Users/AJYou/OneDrive/Desktop/Research_stuff/Bibliography/Bibliography}

\end{document}